\documentclass[12pt]{amsart}
\usepackage{amsmath,amssymb,amsthm, mathtools}
\usepackage[usenames,dvipsnames]{color}
\usepackage{hyperref}  
\usepackage{tikz-cd}
\numberwithin{equation}{section}

\def\O{{\mathcal{O}}}
\def\p{{\mathfrak{p}}}

\def\C{\mathbb{C}}
\def\Z{\mathbb{Z}}
\def\Q{\mathbb{Q}}

\def\F{\mathbb{F}}

\def\ord{{\rm ord}}

\newcommand{\invlim}{\displaystyle \lim_{\leftarrow}}

\newcommand{\inflim}[1]{\displaystyle \lim_{#1 \rightarrow \infty}}

\def\Z{{\mathbb Z}}

\newcommand\rs{\widehat{\mathbb{C}}}

\newcommand{\Fqbar}{\overline{\mathbb{F}_q}}
\newcommand{\kbar}{\overline{k}}
\newcommand{\PP}{\mathbb{P}}
\newcommand\Fq{\mathbb{F}_q}
\newcommand\Fp{\mathbb{F}_p}

\def\O{{\mathcal{O}}}

\newcommand{\m}{\mathfrak{m}}

\newcommand\sphere{\mathbb S^2}

\newcommand\FPP{\text{FPP}}
\newcommand\IMG{\text{IMG}}

\newenvironment{Gal}{{\rm Gal\,}}{}

\newenvironment{Stab}{{\rm Stab \,}}{}
\newenvironment{Aut}{{\rm Aut}}{}

\newenvironment{Per}{{\rm Per}}{}
\newenvironment{Fix}{{\rm Fix}}{}
\newenvironment{img}{{\rm IMG}}{}

\newtheorem{theorem}{Theorem}[section]
\newtheorem{lemma}[theorem]{Lemma}
\newtheorem{proposition}[theorem]{Proposition}
\newtheorem{proposition-definition}[theorem]{Proposition-Definition}
\newtheorem{corollary}[theorem]{Corollary}

\theoremstyle{definition}

\newtheorem{question}[theorem]{Question} 
\newtheorem{definition}[theorem]{Definition}

\theoremstyle{remark}
\newtheorem*{remark}{Remark}

\begin{document}
\title[Iterated monodromy groups and periodic points]{Iterated monodromy groups of rational functions and periodic points over finite fields}
\date{\today}

\author[Bridy]{Andrew Bridy}
\email{andrewbridy@gmail.com}
\address{Departments of Political Science and Computer Science, Yale University, 125 Prospect St, New Haven, CT 06511, USA}

\author{Rafe Jones}
\email{rfjones@carleton.edu}
\address{Department of Mathematics and Statistics, Carleton College, 1 North College St, Northfield, MN 55057, USA}

\author{Gregory Kelsey}
\email{gkelsey@bellarmine.edu}
\address{Department of Mathematics, Bellarmine University, 2001 Newburg Rd., Louisville, KY 40205, USA}

\author[Lodge]{Russell Lodge}
\email{russell.lodge@indstate.edu}
\address{Department of Mathematics and Computer Science, Indiana State University, 200 North Seventh Street, Terre Haute, IN 47809, USA}

\begin{abstract}
Let $q$ be a prime power and $\phi$ a rational function with coefficients in a finite field $\mathbb{F}_q$. For $n \geq 1$, each element of $\mathbb{P}^1(\F_{q^n})$ is either periodic or strictly preperiodic under iteration of $\phi$. Denote by $a_n$ the proportion of periodic elements. Little is known about how $a_n$ changes as $n$ grows, unless $\phi$ is a power map or Chebyshev polynomial. We give the first results on this question for a wider class of rational functions:  $a_n$ has lim inf $0$ when $q$ is odd and $\phi$ is quadratic and neither Latt\`es nor conjugate to a one-parameter family of exceptional maps. We also show that $a_n$ has limit $0$ when $\phi$ is a non-Chebyshev quadratic polynomial with strictly preperiodic finite critical point and $q$ is an odd square. Our methods yield additional results on periodic points for reductions of post-critically finite (PCF) rational functions defined over number fields.


The difficulty of understanding $a_n$ in general is that $\mathbb{P}^1(\F_{q^n})$ is a finite set with no ambient geometry. In fact, $\phi$ can be lifted to a PCF rational map on the Riemann sphere, where we show that $a_n$ is given by counting elements of the iterated monodromy group (IMG) that act with fixed points at all levels of the tree of preimages. Using a martingale convergence theorem, we translate the problem to determining whether certain IMG elements exist. This in turn can be decisively addressed using the expansion of PCF rational maps in the orbifold metric.     
\end{abstract}

\thanks{The authors thank Sarah Koch for suggesting the collaboration. Greg and Russell thank Kevin Pilgrim for fruitful research visits, and Russell acknowledges the generous support of NCTS in Taipei.}

\maketitle

\section{Introduction}

\label{introduction}


Let $\Fq$ denote a finite field of characteristic $p$, with algebraic closure $\Fqbar$. Every $\phi(x) \in \Fq(x)$ acts on $\PP^1(\Fqbar)$, and the orbit of every point under this action is defined over a finite extension of $\Fq$, and hence eventually enters a cycle. This allows us to make a fundamental distinction between two kinds of points in $\PP^1(\Fqbar)$: those that lie in a cycle under $\phi$, which we call \textit{periodic}, and those that do not. For any set $S$ on which $\phi$ is a self-map, denote by $\Per(\phi, S)$ the set of points of $S$ that are periodic under $\phi$.
\begin{question} \label{mainquestion1}
Fix a prime power $q$ and rational function $\phi \in \Fq(x)$ of degree at least two. How does $\#\Per(\phi, \PP^1(\F_{q^{n}}))/(q^n+1)$ vary as $n \to \infty$?
\end{question}


There has been recent interest in questions about the periodic points of mappings in finite fields, partially motivated by an attempt to provide a rigorous analysis of Pollard's famous ``rho method" for integer factorization \cite{pollard}. Despite this, almost nothing is known about a general answer to Question \ref{mainquestion1}, even in a qualitative sense, except for highly constrained mappings such as power maps. Pollard's analysis of the rho method uses the heuristic that the dynamics of specific mappings mimic those of random mappings. A random mapping on a set of size $k$ has $O(\sqrt{k})$ periodic points (see e.g. \cite[Theorem 2]{Flajolet}), so by this heuristic, $\#\Per(\phi, \PP^1(\F_{q^{n}}))/(q^n+1)$ should approach zero as $n$ grows. However, because $\phi$ is a rational function, it must exhibit certain non-random behavior. Crucially, the actions of $\phi$ on $\PP^1(\F_{q^{n}})$ as $n$ varies are not independent of one another. Table \ref{tab:data} presents some data on Question \ref{mainquestion1} for $q = 3$ and $\deg \phi = 2$, and suggests the complexities involved. 

The answer to Question \ref{mainquestion1} is well understood in the case that $\phi$ is a power map or Chebyshev polynomial \cite{manes}. Recent work of Garton \cite{garton2021periodic} sheds some light on the complementary problem of finding $\#\Per(\phi, \PP^1(\F_{q^{n}}))/(q^n+1)$ when $n$ is fixed and $\phi$ varies, while Juul \cite{juul2} studies the size of the image set $\phi^m(\PP^1(\F_{q^{n}}))$ for fixed $m$ as $n$ grows, under certain hypotheses on $\phi$.

\begin{table}
    \centering
    {\renewcommand{\arraystretch}{1.5}
    \begin{tabular}{|c|c|c|c|c|c|c|}
        \hline 
        $n$ & $x^2$ & $x^2-1$ & $x^2-2$ & $\frac{x^2-2}{x^2}$ & $\frac{x^2-2}{x^2-1}$ & $\frac{x^2-1}{x^2}$ \\
        \hline 
        1 &  0.750 & 0.750 & 0.500 & 0.250 & 0.500 & 0.750 \\
        \hline
        2 & 0.300 & 0.500 & 0.400 & 0.300 & 0.200 & 0.500 \\
        \hline
        3 & 0.536 & 0.214 & 0.393 & 0.250 & 0.286 & 0.321 \\
        \hline
        4 & 0.085 & 0.061 & 0.293 & 0.329 & 0.073 & 0.159 \\
        \hline
        5 & 0.504 & 0.299 & 0.377 & 0.250 & 0.254 & 0.176 \\
        \hline
        6 & 0.127 & 0.060 & 0.314 & 0.325 & 0.052 & 0.105 \\
        \hline
        7 & 0.501 & 0.085 & 0.375 & 0.250 & 0.250 & 0.043 \\
        \hline
        8 & 0.032 & 0.017 & 0.266 & 0.315 & 0.023 & 0.046 \\
        \hline
        9 & 0.500 & 0.031 & 0.375 & 0.250 & 0.250 & 0.014 \\
        \hline
        10 & 0.125 & 0.011 & 0.313 & 0.328 & 0.003 & 0.021 \\
        \hline
        
    \end{tabular}
    }
    \bigskip
    
    \caption{$\#\Per(\phi, \PP^1(\F_{3^{n}}))/(3^n+1)$ for various quadratic $\phi \in \F_3(x)$. Note that $x^2 - 2$ is a Chebyshev polynomial and $\frac{x^2 - 2}{x^2}$ is a Latt\`es map.}
    \label{tab:data}
\end{table}

Question \ref{mainquestion1} is in some sense a ``vertical" question, because one moves up a tower of finite fields. A ``horizontal" question of similar flavor may be posed for a rational function defined over a number field $K$. Given $\phi \in K(x)$, for all but finitely many primes $\p$ in the ring of integers $\O_K$ of $K$, one may reduce the coefficients of $\phi$ modulo $\p$ to obtain a morphism $\phi_\p : \mathbb{P}^1(\mathbb{F}_\p) \to \mathbb{P}^1(\mathbb{F}_\p)$ with $\deg \phi = \deg \tilde{\phi}$, where $\mathbb{F}_\p$ is the residue field $\O_K/\p$. Denote by $N(\p)$ the norm of $\p$, so that $1 + N(\p)$ is the size of $\mathbb{P}^1(\mathbb{F}_\p)$.

\begin{question} \label{mainquestion2}
Let $K$ be a number field, and let $\phi \in K(x)$ have degree at least two. How does $\#\Per(\phi_\p, \PP^1(\F_{\p}))/(1 + N(\p))$ vary as $N(\p) \to \infty$?
\end{question}

The known approaches to Questions \ref{mainquestion1} and \ref{mainquestion2} proceed via Galois theory. 
When all the critical points of $\phi$ have independent, infinite orbits, the Galois groups that arise (see Definition \ref{pgimgdef}) are relatively well-understood, and in fact are iterated wreath products in general. This has led to significant progress on Question \ref{mainquestion2} in this case \cite{juul1}. At the other extreme lie $\phi$ for which all critical points have finite orbits, called post-critically finite (PCF). Here the relevant Galois groups are quite different -- they are finitely generated and so far little understood in arithmetic contexts. By definition every $\phi \in \F_q(x)$ is PCF, and this in large part accounts for our collective state of ignorance on Question \ref{mainquestion1}.

However, Galois groups related to PCF rational functions have been studied in some depth in the setting of complex dynamics. In this article we harness ideas from complex dynamics to give results on Question \ref{mainquestion1} for quadratic maps, and to address Question \ref{mainquestion2} in the PCF case.



\begin{theorem} \label{main1}
Let $\Fq$ be a finite field of odd characteristic, and let $\phi(x) \in \Fq(x)$ have degree $2$. Assume that $\phi$ is not a Latt\`es map or M\"obius-conjugate over $\overline{\Fq}$ to a map of the form $(x^2 + a)/(x^2 - (a+2))$ for $a \in \overline{\Fq}$.
Then 
\begin{equation} \label{vertliminf}
\liminf_{n \to \infty} \frac{\#\Per(\phi, \PP^1(\F_{q^{n}}))}{q^{n}+1} = 0.
\end{equation}
\end{theorem}
Indeed we show something slightly stronger (see Theorem \ref{fppreduction}): for every $\epsilon > 0$ there exists $m \geq 1$ such that 
\begin{equation} \label{limepsilonintro}
\frac{\#\Per(\phi, \PP^1(\F_{q^{mk}}))}{q^{mk}+1} < \epsilon
\end{equation}
for sufficiently large integers $k$.  
See Section \ref{exceptions} for a definition of Latt\`es maps over $\Fq$ and a classification of the maps to which Theorem \ref{main1} does not apply. We remark that being $\Fqbar$-conjugate to a map of the form $(x^2 + a)/(x^2 - (a+2))$ is equivalent to having a critical point that maps to a fixed point after two iterations; in particular, this family includes the degree-2 Chebyshev polynomial. Among quadratic maps up to $\Fqbar$-conjugacy, there are eight Latt\`es maps, unless $\Fq$ has characteristic 7 (see Section \ref{exceptions}). None of the maps in Table \ref{tab:data} apart from $x^2 - 2$ and $\frac{x^2 - 1}{x^2}$ is $\overline{\F_3}$-conjugate to a map of the form $(x^2 + a)/(x^2 - (a+2))$.

The equality \eqref{vertliminf} in Theorem \ref{main1} does not hold for all quadratic $\phi \in \Fq(x)$. For the degree-two monic Chebyshev polynomial $\phi(x) = x^2 - 2$, it is shown in \cite{manes} that the lim inf in \eqref{vertliminf} is $1/4$, and indeed a complete accounting of $\#\Per(\phi, \PP^1(\F_{q^{n}}))/(q^{n}+1)$ is given for this map \cite[Theorem 5.6]{manes}. We prove in Theorem \ref{lattesthm} that the lim inf in \eqref{vertliminf} is at least $1/8$ for a certain class of quadratic Latt\`es maps. We suspect that the lim inf is positive for other Latt\`es maps of degree 2, but that the lim inf is zero for non-Chebyshev, non-Latt\`es maps that are $\overline{\F_q}$-conjugate to $(x^2 + a)/(x^2 - (a+2))$. However, our methods do not allow us to prove this at present. 

The integer $m$ in \eqref{limepsilonintro} depends on the constant field extension contained within the splitting fields of $\phi^n(x) - t$ over $\F_q(t)$ (we use $\phi^n$ to denote the $n$th iterate of $\phi$, and take $\phi^0(x) = x$). When $\phi$ is a quadratic polynomial with non-periodic critical point, results of Pink \cite{pink2} imply that $m \leq 2$ for all $\epsilon$, provided that $\phi$ is not conjugate to a Chebyshev polynomial. In fact, when $q$ is a square, $m = 1$ regardless of $\epsilon$, and we obtain: 




\begin{theorem} \label{main12}
Let $\F_q$ be a finite field of odd characteristic, and let $\phi \in \F_q[x]$ have degree 2. Suppose that $q$ is a square and the unique finite critical point of $\phi$ is strictly preperiodic. If $\phi$ is not $\Fqbar$-conjugate to a Chebyshev polynomial, then  
\begin{equation} \label{limzero}
   \lim_{k \to \infty}\frac{\#\Per(\phi, \PP^1(\F_{q^{k}}))}{q^{k}+1} = 0. 
\end{equation}
\end{theorem}

We turn now to Question \ref{mainquestion2}. The principal known results are those in \cite{juul1}, and concern the case where $\phi$ is ``post-critically generic" in the sense that for all $m,n \geq 0$ and all critical points $\gamma$ and $\gamma'$ of $\phi$, we have $\phi^n(\gamma) \neq \phi^m(\gamma')$ unless $m = n$ and $\gamma = \gamma'$. In this case, Theorem 1.3 of \cite{juul1} gives 
\begin{equation} \label{globalliminf}
\liminf_{N(\p) \to \infty} \frac{\#\Per(\phi_\p, \PP^1(\F_{\p}))}{1 + N(\p)} = 0.
\end{equation}

We establish \eqref{globalliminf} for many PCF rational functions. 
To state our result we require two definitions. First, a rational function with coefficients in a field $K$ is \textit{dynamically exceptional}\footnote{In other work, such as \cite{fpfree}, the terminology \textit{exceptional} is used. However, in the arithmetic setting treated in this article, an \textit{exceptional rational function} has a pre-existing, and quite distinct, meaning.} if there is $\Gamma \subset \PP^1(\overline{K})$ with $\phi^{-1}(\Gamma) \setminus C_\phi = \Gamma$, where $C_\phi \subset \PP^1(\overline{K})$ is the set of critical points of $\phi$. Observe that this condition implies that $\Gamma$ contains no critical points of $\phi$, and that $\phi^{-1}(\Gamma)$ consists of $\Gamma$ and a subset of $C_\phi$. Second, let $\phi \in \C(x),$ $P_\phi$ be the post-critical set of $\phi$ (see Definition \ref{postcritdef}), and $z_0 \in \C \setminus P_\phi$. We say $\phi$ has \textit{doubly transitive monodromy} if the monodromy action of $\pi^1((\mathbb{P}^1(\C) \setminus P_\phi), z_0)$ on $\phi^{-1}(z_0)$ is doubly transitive. Equivalently, the Galois group of $\phi(x) - t$ over $\C(t)$ acts doubly transitively on the roots of $\phi(x) - t$ in $\overline{\C(t)}$.

\begin{theorem} \label{main2}
Let $K$ be a number field and let $\phi \in K(x)$ have degree $d \geq 2$. Assume that $\phi$ is PCF and not dynamically exceptional. Then \eqref{globalliminf} is true if any of the following holds: 
\begin{enumerate} 
\item $d$ is prime; 
\item $\phi$ has doubly transitive monodromy;
\item $\phi$ is $\overline{K}$-conjugate to polynomial.
\end{enumerate}
\end{theorem}


The lim inf in \eqref{globalliminf} is not zero for all $\phi$. In \cite[Example 7.2]{juul1}, it is shown that when $\phi = T_d$, the degree-$d$ monic Chebyshev polynomial, the lim inf in \eqref{globalliminf} is $1/4$ when $d$ is a power of 2, $1/2$ when $d$ is a power of an odd prime, and $0$ otherwise.

Questions \ref{mainquestion1} and \ref{mainquestion2} are linked in more than an intuitive sense. By studying a single Galois-theoretic object, we prove Theorems \ref{main1} and \ref{main2} simultaneously.

\begin{definition} \label{pgimgdef}
Let $k$ be a field with algebraic closure $\kbar$ and let $\phi \in k(x)$ have degree $d \geq 2$. Assume that for all $n \geq 1$, $\phi^n(x) = t$ has $d^n$ distinct solutions in an algebraic closure of $\kbar(t)$. The \textit{profinite geometric iterated monodromy group} of $\phi$ over $k$, written $\text{pgIMG}(\phi)/k$, is the inverse limit
as $n \to \infty$ of the Galois groups of $\phi^n(x) - t$ over $\kbar(t)$. 
\end{definition}
The terminology \textit{geometric} in the definition is because the Galois groups are considered over the ground field $\kbar(t)$. One can also consider the Galois groups over $k(t)$, and this object is known as the profinite arithmetic iterated monodromy group of $\phi$. See Section \ref{reductions} for precise definitions and \cite[Section 2]{fpfree} or \cite{pink2} for more discussion. 

Crucially for the considerations in this article, $\text{pgIMG}(\phi)/k$ comes equipped with a natural action on the tree of preimages 
$$
T_k(\phi) := \bigsqcup_{n \geq 0} \phi^{-n}(t) \subset \overline{k(t)},
$$
where $\phi^{-n}(t) = \{\alpha \in \overline{k(t)} : \phi^n(\alpha) = t\}$ for $n \geq 0$ and 
edges are assigned according to the action of $\phi$. 
Let $d = \deg \phi$, and assume that the characteristic of $k$ is either $0$ or does not divide $d$. Then $\phi^n(x) = t$ has $d^n$ distinct solutions in an algebraic closure of $\kbar(t)$,
and hence $T_k(\phi)$ is a complete $d$-ary rooted tree, with root $t$. The action of $\text{pgIMG}(\phi)/k$ on $T_k(\phi)$ comes from the natural action of Galois groups on the roots of polynomials. 

We describe an abstract complete $d$-ary rooted tree as the set $X^*$ of all words in the alphabet $X = \{0, \ldots, d-1\}$, with an edge connecting $vx$ to $v$ for each $v \in X^*$ and $x \in X$. The root of $X^*$ is the empty word. Denote by $X^n$ the set of words in $X$ of length $n$, which gives the $n$th level of $X^*$. 
Let $\Aut(X^*)$ be the set of tree automorphisms, and note that any $G \leq \Aut(X^*)$ has quotient groups $G_n \leq \Aut(X^n)$ for $n \geq 1$ that are the image of the natural restriction maps.
Define the \textit{fixed-point proportion} of $G_n$ to be
\begin{equation} \label{fppdef}
\FPP(G_n) := \frac{\#\{g \in G_n : \text{$g$ fixes at least one element of $T_n$}\}}{\#G_n},
\end{equation}
and the fixed-point proportion of $G$ to be $\lim_{n \to \infty} \FPP(G_n)$.
Observe that the sequence is non-increasing, and hence the limit must exist. Through the action of $\text{pgIMG}(f)/k$ on $T_k(\phi)$, we identify the former with a subgroup of $\Aut(X^*)$. This subgroup is unique up to conjugacy in $\Aut(X^*)$, and in particular $\FPP(\text{pgIMG}(\phi)/k)$ is well-defined.

In Section \ref{reductions} we use the Chebotarev density theorem for function fields to show that if $\Fq$ is a finite field of characteristic $p$, $\phi \in \Fq(x)$ has degree $d$, and $p > d$, then 
\begin{equation*} 
\liminf_{n \to \infty} \frac{\#\Per(\phi, \PP^1(\F_{q^{n}}))}{q^{n}+1} \leq \FPP(\text{pgIMG}(\phi)/\Fq).
\end{equation*}
See Corollary \ref{infsup}. Building on results in \cite{juul1}, we show in Theorem \ref{numfield} that if $K$ is a number field and $\phi \in K(x)$, then 
\begin{equation} \label{fpp2}
\liminf_{N(\p) \to \infty} \frac{\#\Per(\phi_\p, \F_{\p})}{1 + N(\p)} \leq \FPP(\text{pgIMG}(\phi)/\C).
\end{equation}

We appeal to work of Pink \cite{pink1} to show that when $q$ is odd and $\phi \in \Fq(x)$ is quadratic, there is a map $\tilde{\phi} \in \C(x)$ with the same ramification portrait\footnote{This is the natural graph encoding the dynamics and local degrees of the critical orbits of $\phi$. See Section \ref{exceptions} for a precise definition.} as $\phi$, such that $\text{pgIMG}(\phi)/\Fq$ and $\text{pgIMG}(\tilde{\phi})/\C$ have conjugate actions on their respective trees (Theorem \ref{liftability}), and in particular \begin{equation} \label{fpp3}
\FPP(\text{pgIMG}(\phi)/\Fq) = \FPP(\text{pgIMG}(\tilde{\phi})/\C).
\end{equation}

In light of \eqref{fpp2} and \eqref{fpp3}, we study $\text{pgIMG}(f)/\C$ for arbitrary PCF $f \in \C(x)$. Let $P_f$ be the post-critical set of $f$, and $z_0 \in \C \setminus P_f$. The iterated monodromy group of $f$, denoted $\IMG(f)$, is the quotient of the fundamental group $\pi_1((\mathbb{P}^1(\C) \setminus P_f, z_0)$ by the subgroup acting trivially by monodromy on the tree of preimages $T_{f,z_0} \subset \C$ of $z_0$ under $f$ (Definition \ref{imgdef}). Through its action on $T_{f, z_0}$, one can identify $\IMG(f)$ with a subgroup of $\Aut(X^*)$ (even in an explicit way; see Definition \ref{standardactiondef} or \cite[Section 5.2]{nek}), which is unique up to conjugacy in $\Aut(X^*)$. After conjugating if necessary, we may assume 
$$\IMG(f) \subset \text{pgIMG}(f)/\C \subseteq \Aut(X^*).$$
Moreover, $\text{pgIMG}(f)/\C$ is the closure in $\Aut(X^*)$ of $\IMG(f)$ \cite[Proposition 6.4.2]{nek}, and thus both have the same quotients $G_n \leq \Aut(X^n)$. In particular, 
$$\FPP(\text{pgIMG}(f)/\C) = \FPP(\IMG(f)).$$



See Section \ref{subs:IMGs} for details. In light of this, we study $\FPP$ of iterated monodromy groups. The following is our main result in this direction. 

\begin{theorem} \label{maincx1}
Let $f$ be a PCF rational function of degree $d \geq 2$ with coefficients in $\C$, and assume that $f$ is not dynamically exceptional. If either $d$ is prime or $f$ has doubly transitive monodromy,
then $\FPP(\IMG(f)) = 0$. 
\end{theorem}


Crucially for our proof of Theorem \ref{cxthm1}, $\IMG(f)$ is a self-similar, level-transitive, recurrent subgroup of $\Aut(X^*)$ (see Section \ref{subs:wreath} for definitions). 
In the case where $f$ is a PCF polynomial, Theorem 1.1 of \cite{fpfree} proves that $\FPP(\IMG(f)) = 0$. 
To prove Theorem \ref{maincx1}, we must generalize the group-theoretic tools of \cite{fpfree}, which presents considerable technical obstacles.

First, one loses the special element of $\IMG(f)$ that arises from monodromy at infinity. For polynomial $f$, this gives a spherically transitive element in $\IMG(f)$, which is used in \cite{fpfree} to prove the crucial assertion that the fixed-point process associated to $\IMG(f)$ is a martingale. See Section \ref{fpprocess} for definitions. 
To draw the same conclusion for non-polynomial $f$, we show that if $f$ has prime degree or doubly transitive monodromy, then the fixed-point process attached to $\IMG(f)$ is a martingale (Corollaries \ref{fpprocessmain} and \ref{dtranscor}). Indeed, when $d$ is prime Corollary \ref{fpprocessmain} gives the same conclusion for the fixed point process attached to any self-similar, level-transitive subgroup of $\Aut(X^*)$.

Second, once one knows that the fixed-point process of $\IMG(f)$ is a martingale, one can prove $\FPP(\IMG(f)) = 0$ provided that every element of the set 
\begin{equation} \label{n1def}
\mathcal{N}_1 := \{g \in \IMG(f) : \text{$g(w) = w$ and $g|_w = g$ for some $w \in X^*$}\}
\end{equation}
fixes infinitely many ends of $X^*$, i.e. infinite paths through $X^*$ beginning in $X^0$. (See Section \ref{background} for definitions and see Theorem \ref{cxthm1} for the result.)
When $f$ is a polynomial, this last assertion is proved in \cite{fpfree} using a result of Nekrashevych \cite[Corollary 6.10.7]{nek} showing that the actions on $\Aut(X^*)$ of a set of generators for $\IMG(f)$ may be given by the states of a finite automaton satisfying certain strong properties. No equivalent result exists for general rational functions, and indeed until recently very few IMGs have even been computed for non-polynomial rational functions. 

Using tools from complex dynamics, we show:
\begin{theorem} \label{maincx2}
Let $f$ be a PCF rational function of degree $d \geq 2$ with coefficients in $\C$. Then every element of $\mathcal{N}_1$ fixes infinitely many ends of $X^*$ if and only if $f$ is not dynamically exceptional.  
\end{theorem}
See Section \ref{sec:IMGofPCF}. The main ingredient in the proof of Theorem \ref{maincx2} is the fact that a PCF $f \in \C(x)$ is subhyperbolic, i.e., expanding (in some orbifold metric) away from post-critical periodic points. This expansion forces lifts of loops under iterates of $f$ to contract, which imposes strong conditions on elements of $\mathcal{N}_1$. In particular, an element of $\mathcal{N}_1$ that fixes only finitely many ends of $X^*$ must be a loop encircling (in $\rs \setminus P_f$) a single repelling periodic point in $P_f$, and moreover every backward orbit of this point must either remain in $P_f$ or contain a critical point. This forces $f$ to be dynamically exceptional.

\section{Dynamically exceptional rational functions over finite fields} \label{exceptions}

In this section we study the exceptions to Theorem \ref{main1}. In particular, we discuss Latt\`es maps over finite fields and give a characterization of dynamically exceptional quadratic rational functions over an arbitrary field of characteristic $\neq 2$.

Recall from Section \ref{introduction} that a rational function with coefficients in a field $K$ is dynamically exceptional if there is $\Gamma \subset \PP^1(\overline{K})$ with $\phi^{-1}(\Gamma) \setminus C_\phi = \Gamma$, where $C_\phi \subset \PP^1(\overline{K})$ is the set of critical points of $\phi$.
In this section we study dynamically exceptional rational functions of degree $2$ over an arbitrary field of characteristic different from 2. 

Let $K$ be a field with fixed algebraic closure $\overline{K}$, and let $\phi \in K(x)$. For $\alpha \in \PP^1(\overline{K})$ with $\alpha \neq \infty$ and $\phi(\alpha) \neq \infty$, the  \textit{ramification index} $e_\phi(\alpha)$ of $\phi$ at $\alpha$ is the multiplicity of $\alpha$ as a root of the numerator of $\phi(x) - \phi(\alpha)$. 
If $\alpha = \infty$ or $\phi(\alpha) = \infty$, then $e_\phi(\alpha) = e_{\mu \circ \phi \circ \mu^{-1}} (\mu(\alpha))$, where $\mu$ is a Mobius transformation mapping both $\alpha$ and $\phi(\alpha)$ away from infinity. We call $\alpha$ a critical point for $\phi$ if $e_\phi(\alpha) > 1$. 


Define the \textit{ramification portrait} of $\phi$ to be the edge-labeled directed graph whose vertex set is the union of the orbits of all critical points of $\phi \in \PP^1(\overline{K})$, and where each vertex $\alpha$ has an arrow to $\phi(\alpha)$ with label $e_\phi(\alpha)$. Note that the graph is not vertex-labeled, so we do not record the specific points involved.

For instance, if $K$ has characteristic not equal to 2, then $\phi(x) = (x^2 - 2)/x^2$ has critical points $0$ and $\infty$, with $0 \to \infty \to 1 \to -1 \to -1$. This gives ramification portrait
$\bullet \xrightarrow{2} \bullet \xrightarrow{2} \bullet\rightarrow \bullet \circlearrowleft$. Because we deal here with quadratic maps, and so every critical point $\alpha$ has $e_\phi(\alpha) = 2$, we rewrite this as 
\begin{equation} \label{lattes444}
\bullet \rightarrow \bullet \rightarrow \circ \rightarrow \circledcirc,
\end{equation} 
where $\bullet$ denotes a critical point, $\circ$ a non-critical point, and $\circledcirc$ a non-critical fixed point. Denote a critical fixed point by $\odot$. As another example, if $\phi$ is the degree-2 Chebyshev polynomial $x^2 - 2$, then $\phi$ has ramification portrait
\begin{equation} \label{chebram}
\odot \qquad \bullet \rightarrow \circ \rightarrow \circledcirc,
\end{equation} 
We note that the ramification portraits in \eqref{lattes444} and \eqref{chebram} uniquely determine $\phi$ up to Mobius conjugation. 

The next definition is used throughout the remainder of the paper. 
\begin{definition} \label{postcritdef}
Let $K$ be a field and $\phi \in K(x)$. Let $\gamma_1, \ldots, \gamma_j$ be the critical points of $\phi$, which lie in  $\mathbb{P}^1(\overline{K})$. The \textit{post-critical} set of $\phi$ is
$$
P_\phi := \bigcup_{i = 1}^j \bigcup_{k \geq 1} \phi^k(\gamma_i) \subset \mathbb{P}^1(\overline{K}).
$$
\end{definition}

For the purposes of this article, we define $\phi \in K[x]$ to be a \textit{Latt\`es map} if there exists a function $r : \PP^1(\overline{K}) \to \Z$ such that 
\begin{equation} \label{lattesdef}
\text{$r(\phi(\alpha)) = e_\phi(\alpha)r(\alpha)$ and $r(\alpha) = 1$ outside of $P_\phi$}.
\end{equation}
When $K$ is a finite field, these are precisely the liftable maps that lift to Latt\`es maps defined over $\C$ (see Section \ref{reductions} for a definition of lifting). This is because over $\C$, the existence of the function $r$ is equivalent to the usual definition of Latt\`es maps as given by a finite quotient of a self-map of an elliptic curve; see \cite[Theorem 4.1]{milnor}.

\begin{proposition} \label{quadraticlattesportraits}
Let $K$ be a field of characteristic not equal to 2, and let $\phi \in K(x)$ have degree 2. Then $\phi$ is a Latt\`es map if and only if the ramification portrait of $\phi$ is the one in \eqref{lattes444} or one of the following:
\begin{equation} \label{lattestypes1}
\bullet \rightarrow \circ \rightarrow \circledcirc \; \; \; \bullet \rightarrow \circ \rightarrow \circledcirc, \quad \quad \quad \bullet \rightarrow \circ \rightarrow \circ \rightleftarrows \circ \leftarrow \circ \leftarrow \bullet
\end{equation}
\begin{equation} \label{lattestypes2}
\begin{tikzcd}[row sep=small, column sep=small]
\bullet \arrow[r] &  \circ \arrow[r] & \circ \arrow[d] & \circ \arrow[l] & \bullet \arrow[l]\\
 & & \circledcirc & &
\end{tikzcd}
\end{equation}
\end{proposition}

\begin{proof}
Let $\Delta = \{\alpha \in \PP^1(\overline{K}) : r(\alpha) > 1\}$. By definition of $r$, we have $\Delta = P_\phi$ and $\phi^{-1}(\Delta) = \Delta \cup C_\phi$. Thus
\begin{equation} \label{preims}
2\#\Delta = \sum_{\alpha \in \phi^{-1}(\Delta)} e_\phi(\alpha) \leq \#\Delta + 2\#C_\phi,
\end{equation}
with equality if and only if $\Delta$ and $C_\phi$ are disjoint. Because $K$ has characteristic not equal to 2, $\#C_\phi = 2$, and we conclude from \eqref{preims} that $\#\Delta \leq 4$, with equality if and only if $\Delta \cap C_\phi = \emptyset$.

Suppose that $\#\Delta < 4$, and let $\gamma \in \Delta \cap C_\phi$. Observe that $\phi^{-1}(\gamma) \subset C_\phi \cup P_\phi$ and thus if $\phi^{-1}(\gamma)$ contains no critical points, then $\phi^{-1}(\gamma)$ consists of two post-critical points. But there is only one critical point of $\phi$ besides $\gamma$, so it is impossible for both points in $\phi^{-1}(\gamma)$ to be post-critical. Hence $\phi^{-1}(\gamma)$ consists of a critical point. Now $\gamma$ cannot be periodic, for otherwise $r(\gamma)$ is not well-defined. Hence if $\phi(\gamma)$ is periodic, then it is a fixed point. But then $\phi^{-1}(\phi(\gamma))$ contains both $\gamma$ and $\phi(\gamma)$, which is impossible. Hence $\phi(\gamma)$ cannot be periodic. Because $\#\Delta \leq 3$, it must be the case that $\phi^2(\gamma)$ is a fixed point, and we have ramification portrait \eqref{lattes444}.

Suppose now that $\#\Delta = 4$, and thus $P_\phi \cap C_\phi = \emptyset$. Because $\phi$ cannot have a periodic critical point, $P_\phi$ must contain a cycle, and for each $\alpha$ in this cycle, $\phi^{-1}(\alpha)$ cannot contain a critical point, as otherwise $\phi^{-1}(\alpha)$ consists only of a critical point, which must then be periodic. It follows that the length of this cycle can be at most 2. If $P_\phi$ contains a 2-cycle, one easily checks that the only possible ramification portrait is the second one in \eqref{lattestypes1}.

Now a fixed point in $P_\phi$ cannot have a pre-image that is a critical point, and hence $P_\phi$ can contain at most two fixed points. If there are exactly two, then we must have the first ramification portrait in \eqref{lattestypes1}. If there is only one, then we must have the ramification portrait in \eqref{lattestypes2}.
\end{proof}

We now describe quadratic Latt\`es maps over a field of characteristic not equal to 2. We use the normal form $\phi(x) = (x^2 + a)/(x^2 + b), a \neq b$, which exists for every degree-2 rational function except those conjugate over $\overline{K}$ to $x^{\pm 2}$, and can be obtained by conjugating a map's two critical points to $0$ and $\infty$, and then conjugating again so $\phi(\infty) = 1$. We observe that this conjugation is defined over $K$ if and only if the map's critical points lie in $K$; otherwise the conjugation is over a quadratic extension of $K$. The normal form is unique except that if $ab \neq 0$, then conjugation by $x \mapsto a/(bx)$ takes $(x^2 + a)/(x^2 + b)$ to $(x^2 + (a^2/b^3))/(x^2 + (a/b^2))$. This is the normal form found in \cite{pink1}, and is related to the normal form for critically marked quadratic rational functions given in \cite[Section 6]{milnorquadratic}. 

\begin{proposition} \label{quadraticlattesclass}
If $K$ is a field of characteristic not equal to 2, then every degree-2 Latt\`es map is conjugate (over $\overline{K}$) to one of the following:
\begin{equation} \label{quadraticlattesmaps}
\frac{x^2-2}{x^2}, \qquad \frac{x^2+\alpha_1}{x^2-\alpha_1}, \qquad \frac{x^2+\alpha_2}{x^2-\alpha_2}, \qquad \frac{x^2 + \alpha_3}{x^2 - (\alpha_3 + 2)}, \qquad \frac{x^2 + \frac{1}{\alpha_3}}{x^2 - \frac{1}{\alpha_3+2}},
\end{equation}
where $\alpha_1$ is a root of $y^2 + 1$ (in $\overline{K})$, $\alpha_2$ is a root of $y^2 - 2y - 1$, and $\alpha_3$ is a root of $y^2 + 5y + 8$. 
\end{proposition}
\begin{remark}
The two maps $\frac{x^2+\alpha_2}{x^2-\alpha_2}$, where $\alpha_2$ is either root of $y^2 - 2y - 1$, are in fact conjugate to each other by $x \mapsto -1/x$. Otherwise, no two maps in \eqref{quadraticlattesmaps} are conjugate. Hence there are 8 conjugacy classes of Latt\`es maps (over $\overline{K}$) if $K$ has characteristic not equal to 7. If $K$ has characteristic 7, then $y^2 + 5y + 8$ has only one root in $\overline{K}$, and hence there are only 6 conjugacy classes of Latt\`es maps. 
\end{remark}

\begin{proof}
Let $\phi \in K(x)$ be a degree-2 Latt\`es map. It follows from Proposition \ref{quadraticlattesportraits} that $\phi$ is not conjugate to $x^{\pm 2}$, and hence we may write $\phi(x) = (x^2 + a)/(x^2 + b)$ for some $a,b \in \overline{K}$ with $a \neq b$. Each of the ramification portraits described in Proposition \ref{quadraticlattesportraits} then gives rise to two polynomial conditions on $a$ and $b$. For instance, the portrait in \eqref{lattestypes2} forces $\phi^2(0) = \phi^2(\infty)$, which implies $b = -a$. The same portrait implies $\phi^4(\infty) = \phi^3(\infty)$, which gives
$(a^2 + 1)(a^2 - 2a - 1)=0$. The ramification portrait \eqref{lattes444} leads to the first map in \eqref{quadraticlattesmaps}, and the portraits in \eqref{lattestypes1} lead to the fourth and fifth maps in \eqref{quadraticlattesmaps}, respectively. 
\end{proof}


We now give our characterization of dynamically exceptional quadratic rational functions. 
 
\begin{proposition} \label{dynamically exceptional characterization}
Let $K$ be a field of characteristic $\neq 2$, and let $\phi \in K(x)$ have degree 2. Then $\phi$ is dynamically exceptional if and only if $\phi$ is a Latt\`es map or conjugate over $\overline{K}$ to $(x^2 + a)/(x^2 - (a+2))$ for some $a \in \overline{K}$.
\end{proposition}

\begin{remark}
Maps conjugate to the degree-2 Chebyshev polynomial, as well as Latt\`es maps with ramification portrait \eqref{lattestypes1}, are conjugate to $(x^2 + a)/(x^2 - (a+2))$ for appropriate $a \in \overline{K}$.
\end{remark}

\begin{proof}
By definition, there is $\Gamma \subset \PP^1(\overline{K})$ with $\phi^{-1}(\Gamma) \setminus C_\phi = \Gamma$. This implies that $\Gamma \subseteq \phi^{-1}(\Gamma)$ and $\Gamma \cap C_\phi = \emptyset$. Hence 
\begin{equation} \label{excepfact}
2\#\Gamma = \sum_{\alpha \in \phi^{-1}(\Gamma)} e_\phi(\alpha) = \#\Gamma + 2\#(\phi^{-1}(\Gamma) \cap C_\phi),    
\end{equation}
and it follows that $\#\Gamma \in \{2,4\}$, according to whether $\#(\phi^{-1}(\Gamma) \cap C_\phi)$ is $1$ or $2$. 

First suppose that $\#\Gamma = 2$ and $\phi^{-1}(\Gamma)$ contains a single critical point $c$. Because $\phi(\Gamma) \subseteq \Gamma$, $c$ cannot be periodic, for then $c \in \Gamma$. Similarly, $\phi(c)$ cannot be periodic, for then its unique preimage $c$ must be periodic as well. Thus $\phi^2(c)$ is a fixed point for $\phi$, and after conjugation we may assume $c = \infty$, $\phi(c) = 1$, and $\phi^2(c) = -1$, giving the map $(x^2 + a)/(x^2 - (a+2))$ for some $a \in \overline{K}$. We remark that any map with ramification portrait \eqref{lattestypes1} or \eqref{chebram}, and hence any map conjugate to the degree-2 Chebyshev polynomial, is a special case.

Now suppose that $\#\Gamma = 4$, and $\phi^{-1}(\Gamma)$ contains both critical points of $\phi$, i.e., $\phi^{-1}(\Gamma) = \Gamma \cup C_\phi$. Then we may define a function $r : \PP^1(\overline{K}) \to \Z$ satisfying \eqref{lattesdef} by taking $r(\alpha) = 2$ for $\alpha \in \Gamma$ and $r(\alpha) = 1$ for $\alpha \not\in \Gamma$. Hence $\phi$ is a Latt\`es map. 
\end{proof}

In general we expect a Latt\`es map $\phi$ defined over a finite field $\F_q$ to satisfy $$\liminf_{n \to \infty} \frac{\#\Per(\phi, \PP^1(\F_{q^{n}}))}{q^n+1} > 0,$$
much as happens with Chebyshev polynomials \cite{manes}. Using work of Ugolini \cite{ugolini1}, we prove this happens in a certain case:

\begin{theorem} \label{lattesthm}
Let $K = \F_p$ with $p \equiv 1 \bmod{4}$, and suppose that $\phi$ is conjugate over $K$ to the Latt\`es map $\frac{x^2+a}{x^2-a}$, where $a \in K$ and $a^2 + 1 = 0$. Then
\begin{equation} \label{prop}
    \liminf_{n \to \infty} \frac{\#\Per(\phi, \PP^1(\F_{p^{n}}))}{p^n+1} \geq \frac{1}{8}
\end{equation}
\end{theorem}

\begin{remark}
There are $\phi$ that are $\overline{K}$-conjugate to $\frac{x^2+a}{x^2-a}$, where $a \in K$ with $a^2 + 1 = 0$, but not $K$-conjugate to any such map. Indeed, if $\phi$ is $\overline{K}$-conjugate to a map of this kind, then it is $K$-conjugate to such a map if and only if its critical points lie in $K$.  
\end{remark}

\begin{proof}
Because $\phi$ is conjugate over $K$ to a map whose critical points are defined over $K$, the critical points of $\phi$ must be defined over $K$. Applying a conjugacy that moves these critical points to $\pm 1$, we see that $\phi$ is conjugate over $K$ to $\psi(x) = k(x + x^{-1})$, where $k^2 + \frac{1}{4} = 0$. 
As detailed in \cite[Section 3]{ugolini1}, the map $\psi$ descends from a degree-2 endomorphism on the elliptic curve $y^2 = x^3 + x$ defined over $\F_p$, which has endomorphism ring $R:=\Z[i]$. Moreover, because $p \equiv 1 \bmod{4}$, the two degree-2 maps in $R$, namely $[1 \pm i]$, are both defined over $\F_p$, and indeed have the form $(x,y) \mapsto (\psi(x), y \tau(x))$ with $\tau(x) = c(x^2 - 1)/x^2 \in \Fp(x)$. 

Our analysis of the action of $\psi$ on $\PP^1(\F_{p^n})$ begins by partitioning $\PP^1(\F_{p^n})$ into two $\psi$-invariant sets which, by the Hasse bound, have approximately equal size when $p^n$ is large. Let $S$ be the three roots of $x^3 + x$, which lie in $\Fp$ since $p \equiv 1 \bmod{4}$. Set
\begin{align*}
 A_n = \begin{cases} \{x \in \F_{p^n} : \text{there is $y \in \F_{p^n}$ with $(x,y) \in E(\F_{p^n})$}\} \cup \{\infty\} & \text{if $\sqrt{2} \in \F_{p^n}$} \\ 
 \{x \in \F_{p^n} : \text{there is $y \in \F_{p^n}$ with $(x,y) \in E(\F_{p^n})$}\} \setminus S & \text{if $\sqrt{2} \not\in \F_{p^n}$}
 \end{cases}
 \end{align*}
 and take $B_n = \PP^1(\F_{p^n}) \setminus A_n$. 

 Because endomorphisms of $E$ preserve $E(\F_{p^n})$, we immediately have $\psi(A_n) \subseteq A_n$ if $\sqrt{2} \in \F_{p^n}$. If $\sqrt{2} \not\in \F_{p^n}$ and $\alpha \in A_n$, then $\psi(\alpha) \in A_n$ unless $\psi(\alpha) \in S$. But $\psi^{-1}(S) = S \cup \{\pm 1\}$, and $\pm 1 \not\in A_n$ since $\sqrt{2} \not\in \F_{p^n}$. Thus $\psi^{-1}(S) \cap A_n = \emptyset$.
 Suppose now that $\alpha \in B_n$, and let $\beta$ satisfy $(\alpha, \beta) \in E(\overline{\F}_{p})$. The $y$-coordinate of $[1\pm i](\alpha,\beta)$ has the form $\beta \tau(\alpha)$. But $\tau(\alpha) \in \F_{p^n}$, so $\beta \tau(\alpha) \in \F_{p^n}$ if and only if $\beta \in \F_{p^n}$ or $\tau(\alpha) = 0$ (i.e. $\alpha = \pm 1$). If $\sqrt{2} \in \F_{p^n}$, then $\{\pm 1\} \cap B_n = \emptyset$, whence $\psi(B_n) \subseteq B_n$. If $\sqrt{2} \not\in \F_{p^n}$, then the entire orbits of $\pm 1$ under $\psi$ are contained in $B_n$, and so again we have $\psi(B_n) \subseteq B_n$.
 
  If we put $f(n) = (\#A_n)/(p^n+1)$ and $g(n) = (\#B_n)/(p^n+1)$, then the Hasse bound implies that both $f(n)$ and $g(n)$ are $1/2 + O(p^{-n/2})$. In particular, 
 \begin{equation} \label{hasse}
\lim_{n \to \infty} \frac{\#A_n}{p^n+1} = \lim_{n \to \infty} \frac{\#B_n}{p^n+1} = \frac{1}{2}.
 \end{equation}

Let $\pi_p \in R$ denote the Frobenius endomorphism of $E$ (which is given explicitly by $(r + \sqrt{r^2 - 4p})/2$ where $r = p + 1 - \#E(\F_p)$), and let $\p$ be the ideal $(1+i)$ of $R$. Theorem 3.5 of \cite{ugolini1} implies that each periodic point in $A_n$ (resp. $B_n$) is the root of a complete binary rooted tree whose depth is given by $v_\p(\pi_p^n - 1)$ (resp. $v_\p(\pi_p^n + 1))$, where $v_\p$ denotes the $\p$-adic valuation. The only exception is the fixed point at $\infty$, whose tree includes the critical points $\pm 1$ but otherwise is a complete binary tree with depth given as in the previous sentence. 
We have
$$2 = v_\p(2) = v_\p((\pi_p + 1) - (\pi_p - 1)) \geq \min\{v_\p(\pi_p+1), v_\p(\pi_p-1)\},
$$
and it follows that either $A_n$ or $B_n$ is composed of periodic points for $\psi$, each one mapped to by a binary tree of non-periodic points of depth at most 2. Without loss of generality, say that $A_n$ satifies this condition. Then 
\begin{equation} \label{lattespdc}
\liminf_{n \to \infty} \frac{\#\Per(\psi, A_n)}{\#A_n} \geq 1/4.
\end{equation}
Combining \eqref{hasse} and \eqref{lattespdc} gives
\begin{align*}
\liminf_{n \to \infty}  \frac{\#\Per(\psi, \PP^1(\F_{p^{n}}))}{p^n+1} & \geq \liminf_{n \to \infty} \frac{\#\Per(\psi, A_n)}{p^n+1} \\
& = \liminf_{n \to \infty} \left( \frac{\#\Per(\psi, A_n)}{\#A_n} \cdot \frac{\#A_n}{p^n+1} \right) \\
& \geq \frac{1}{4} \cdot \frac{1}{2} = \frac{1}{8}.
\end{align*}
\end{proof}

To illustrate the results of this section, we give some further discussion of the maps in Table \ref{tab:data}, which gives data for $K = \F_3$ and all quadratic maps $\phi(x) = (x^2-a)/(x^2-b)$ with $a,b \in K$. The cases $(a,b) = (1,2)$ and $(a,b) = (2,1)$ produce maps that are conjugate over $\F_3$ and thus have the same dynamics on $\F_3^n$, while all other choices of $(a,b)$ with $a \neq b$ yield maps that are not conjugate over $\overline{\F_3}$. Taking $(a,b) = (0,2)$ gives a map with ramification portrait \eqref{chebram}, which is thus $\overline{\F_3}$-conjugate to the degree-2 Chebyshev polynomial $x^2 - 2$. Taking $(a,b) = (2,0)$ gives a Latt\`es map with ramification portrait \eqref{lattestypes1}. Taking $(a,b) \in \{(0,1), (1,2), (1,0)\}$ gives a map that is not dynamically exceptional. We note that $(a,b) = (0,2)$ gives a map conjugate to $x^2 - 1$. Table \ref{tab:data} shows $\#\Per(\phi, \PP^1(\F_{3^{n}}))/3^n$ for $n \leq 10$, and also includes the map $x^2$, whose periodic points in $\PP^1(\F_{3^{n}})$ are the same as $1/x^2$.

\section{Reducing Theorems \ref{main1} and \ref{main2} to statements about IMGs} \label{reductions}

In this section we show that to prove Theorems \ref{main1} and \ref{main2}, it is enough to prove Theorem \ref{maincx1}.



For each $n \geq 1$, let $K_n^\text{arith}$ be the extension of $\Fq(t)$ obtained by adjoining the roots of $\phi^n(x) - t$, and 
$K_n^\text{geom}$ the extension of $\overline{\Fq}(t)$ obtained by adjoining the roots of $\phi^n(x) - t$ (recall our standing assumption that $\phi^n(x) - t$ has $d^n$ distinct roots in $\overline{\Fq(t)}$). We note that $K_n^{\text{geom}}$ is equal to the compositum $K_n^{\text{arith}}\overline{\Q}$, which in turn is equal to the compositum $K_n^{\text{arith}}\overline{\Q}(t)$.

Denote by $G_n$ the Galois group of $K_n^\text{geom}$ over $\overline{\Fq}(t)$, and note that $G_n$ is the natural quotient of $\text{pgIMG}(\phi)/\F_q$ ($= \invlim G_n$) obtained by restricting its action on $T(\phi)$ to the set $T_n(\phi)$ of vertices having distance $n$ from the root of $T(\phi)$. The first main result of this section relates $\FPP(G_n)$ to certain counts of periodic points.

\begin{theorem} \label{fppreduction}
Let $\Fq$ be a finite field of characteristic $p$ and $\phi \in \Fq(x)$ have degree $d$ with $2 \leq d < p$. Let $n \geq 1$ and let $K_n^{\text{arith}} \cap \overline{\Fq} = \F_{q^m}$, so that $\F_{q^m}$ is the maximal constant field subextension of $K_n^{\text{arith}}$. Then for every $\delta > 0$ there is a constant $k_0$ such that
\begin{equation} \label{limepsilon}
\frac{\#\Per(\phi, \PP^1(\F_{q^{mk}}))}{q^{mk}+1} <
\FPP(G_n) + \delta
\end{equation}
for all $k > k_0$. 
\end{theorem}

To prove Theorem \ref{fppreduction}, we begin with two elementary lemmas, the first of which is Lemma 5.2 of \cite{juul1}. 

\begin{lemma}[\cite{juul1}] \label{perim}
If $f$ is a function acting on a finite set $\mathcal{U}$, then $\text{Per}(f, \mathcal{U}) = \bigcap_{n \geq 0} f^n(\mathcal{U})$. In particular $\#\text{Per}(f, \mathcal{U}) \leq \#f^n(\mathcal{U})$ for every $n \geq 0$.
\end{lemma}

We say that the degree of $\beta \in \F_q$, written $\deg \beta$, is the degree of the minimal polynomial of $\beta$ over $\F_q$. 
\begin{lemma} \label{ffcount}
Let $\F_q$ be a finite field with $q$ elements, and let $k > 1$ be an integer. Then 
\begin{equation} \label{smalldeg}
    \#\{\beta \in \F_{q^k} : \deg \beta < k\} \leq 2q^{k/2}.
\end{equation}
\end{lemma}

\begin{proof}
The subfields of $\F_{q^k}$ are precisely $\F_{q^r}$ for $r \mid k$, and $\F_{q}(\beta) = \F_{q^{\deg \beta}}$. Thus $\#\{\beta \in \F_{q^k} : \deg \beta < k\}$ is bounded above by $\sum_{r \mid k, r \neq k} q^r$, and
$$
\sum_{r \mid k, r \neq k} q^r \leq q^{k/2} + q^{(k/2) - 1} + q^{(k/2) - 2} + \cdots = q^{k/2} \left(1 + \frac{1}{q} + \frac{1}{q^2} + \cdots \right) \leq 2q^{k/2}.
$$
\end{proof}

\begin{proof}[Proof of Theorem \ref{fppreduction}]
Begin by observing that $\#\phi^n(\PP^1(\F_{q^{mk}})) \leq \phi^n(\F_{q^{mk}}) + 1$, and so 
\begin{equation} \label{p1red}
\frac{\#\phi^n(\PP^1(\F_{q^{mk}}))}{q^{mk}+1} \leq \frac{\#\phi^n(\F_{q^{mk}})}{q^{mk}} + \frac{1}{q^{mk} + 1}. 
\end{equation}

We will bound $\frac{\#\phi^n(\F_{q^{mk}})}{q^{mk}}$ for sufficiently large $k$. To do so, we study the extension $K_n^{\text{arith}}/\F_{q^m}(t)$. Because $\F_{q^m}$ is the maximal constant field subextension of $K_n^{\text{arith}}$, we have $\Gal(K_n^{\text{arith}}/\F_{q^m}(t)) = G_n$.


Each place $P$ of $\F_{q^m}(t)$ (resp. $\p$ of $K_n^{\text{arith}}$), has a corresponding discrete valuation $v_P$ (resp. $v_\p$), and we denote by $\O_P$ (resp. $\O_\p$) the ring of integers $\{z \in \F_{q^m}(t)^* : v_P(z) \geq 0\}$ (resp. $\{z \in K_n^{\text{arith}*} : v_\p(z) \geq 0\}$) and we denote by $\m_P$ (resp. $\m_\p$) the maximal ideal $\{z \in \O_P : v_P(z) > 0\}$ (resp. $\{z \in \O_\p : v_p(z) > 0\}$). We denote the residue fields $\O_P/\m_P$ and $\O_\p/\m_\p$ by $\F_P$ and $\F_\p$, respectively, and we denote the canonical maps $\O_P \to \F_P$ and $\O_\p$ to $\F_\p$ by $\pi_P$ and $\pi_\p$, respectively. 

Let $\alpha_1, \ldots, \alpha_{d^n}$ be the roots of $\phi^n(x)-t$ in $\overline{\F_{q^m}(t)}$, and observe that these are all distinct. Let $T$ be the set of places $P$ of $\F_{q^m}(t)$ satisfying all of the following:
\begin{enumerate}
    \item $P$ is not ramified in $K_n^{\text{arith}}$;
    \item every extension $\p$ of $P$ to $K_n^{\text{arith}}$ satisfies $v_{\p}(\alpha_i - \alpha_j) = 0$ for all $i \neq j$;
    \item every extension $\p$ of $P$ to $K_n^{\text{arith}}$ satisfies $v_{\p}(\alpha_i) \geq 0$ for all $i$;
    \item $P$ is not the place at infinity.
\end{enumerate}   
(We remark that condition (2) implies condition (1), though we do not need that for the proof.) Let $P \in T$, and let $\p$ be an extension of $P$ to $K_n^{\text{arith}}$. Condition (4) ensures there is an irreducible polynomial $p(t) \in \F_{q^m}[t]$ of some degree $k \geq 1$ such that $v_P$ is given by $\ord_p(\cdot)$. 
In particular, $\F_P = \F_{q^m}[t]/(p(t))$, which is a finite field of $q^{mk}$ elements. 
Moreover, condition (3) ensures $\alpha_i \in \O_\p$ for all $i = 1, \ldots , d^n$, and condition (2) ensures  
\begin{equation} \label{distinctroots}
\text{$\pi_\p : \{\alpha_1, \ldots, \alpha_{d^n}\} \to \{\pi_p(\alpha_1), \ldots, \pi_\p(\alpha_{d^n})\}$ is a bijection.}
\end{equation}
Let $D(\p/P) \subset G_n$ be the decomposition group of $\p$, i.e. 
$$\{g \in G_n : \text{$v_\p(g(z)) = v_\p(z)$ for all $z \in K_n^{\text{arith}} \setminus \{0\}$}\}.$$
 Observe that any $g \in D(\p/P)$ gives a map $\O_\p \to \O_\p$ that descends to  $\overline{g} \in \Gal(\F_\p/\F_P)$ given by $\overline{g}(z + \p) = g(z) + \p$. For any $\alpha_i$, we have
$$
\pi_\p(g(\alpha_i)) = g(\alpha_i) + \p = \overline{g}(\alpha_i + \p) = \overline{g}(\pi_\p(\alpha_i)),
$$
and it follows from \eqref{distinctroots} that $g$ permutes $\{\alpha_1, \ldots, \alpha_{d^n}\}$ in the same way that $\overline{g}$ permutes $\{\pi_\p(\alpha_1), \ldots, \pi_\p(\alpha_{d^n})\}$. 

Because $\phi$ is defined over $\F_{q^m}$, it commutes with $\pi_\p$, so we have
\begin{equation} \label{preimeq}
\pi_\p(t) = \pi_\p(\phi^n(\alpha_i)) = \phi^n(\pi_\p(\alpha_i)),
\end{equation}
whence $\{\pi_\p(\alpha_1), \ldots, \pi_\p(\alpha_{d^n})\}$ are the preimages of $\pi_\p(t)$ under $\phi^n$. Now $\pi_\p(t)$ is a root of $p(t)$ in $\F_\p$, and hence lies in $\F_P$, since the latter is $\O_P/(p(t))$. Let $\beta = \pi_\p(t)$, and let $\beta'$ be any other root of $p(t)$ in $\F_P$. Then there is $\sigma \in \Gal(\F_\p/\F_P)$ with $\sigma(\beta) = \beta'$. Now $\sigma$ commutes with $\phi$, and so applying $\sigma$ to \eqref{preimeq} shows that the preimages of $\beta'$ under $\phi^n$ are 
$\{\sigma(\pi_\p(\alpha_1)), \ldots, \sigma(\pi_\p(\alpha_{d^n}))\}$. Moreover, $\Gal(\F_\p/\F_P)$ is abelian, and so $\overline{g}$ and $\sigma$ commute for any $g \in D(\p/P)$. It follows that $g$ has a fixed point in $\{\pi_\p(\alpha_1), \ldots, \pi_\p(\alpha_{d^n})\}$ if and only if it has a fixed point in $\{\sigma(\pi_\p(\alpha_1)), \ldots, \sigma(\pi_\p(\alpha_{d^n}))\}$.


Still assuming that $P \in T$, condition (1) implies that the map $g \to \overline{g}$ gives an isomorphism $D(\p/P) \to \Gal(\F_\p/\F_P)$ \cite[Theorem 9.6]{Rosen}. The inverse image of the Frobenius map $x \mapsto x^{q^{mk}}$ is denoted $\text{Frob}(\p/P)$, and the set $\{\text{Frob}(\p/P) : \text{$\p$ extends $P$}\}$ is a conjugacy class of $G_n$ \cite[Proposition 9.7]{Rosen}, which we denote $\text{Frob}(P)$. Observe that if $\text{Frob}(\p/P)$ fixes one of the $\alpha_i$ for some extension $\p$ of $P$, then so does every element of $\text{Frob}(P)$.



Now $\text{Frob}(\p/P)(\alpha_i) = \alpha_i$ is equivalent to $(\pi_p(\alpha_i))^{q^{mk}} = \pi_\p(\alpha_i)$, which is equivalent to $\alpha_i \in \F_{q^{mk}}$.  Thus if $\beta_1, \ldots, \beta_k$ are the roots in $\F_{q^{mk}}$ of $p(t)$, we have
\begin{multline} \label{Frobeq}
\text{Frob$(P)$ acts on $\{\alpha_1, \ldots, \alpha_{d^n}\}$ with at least one fixed point} \\ \Longleftrightarrow 
\text{for every $j \in \{1, \ldots, k\}$, there is $y \in \F_{q^{mk}}$ with $\phi^n(y) = \beta_j$}
\end{multline}
Observe that the latter condition in \eqref{Frobeq} is equivalent to $\{\beta_1, \ldots, \beta_k\} \subset \phi^n(\F_{q^{mk}})$. 

Let $U = \{\text{places $P$ of $\F_{q^m}(t)$ that are unramified in $K_n^{\text{arith}}$}\}$. The Chebotarev Density Theorem for function fields (see e.g. \cite[Theorem 9.13B]{Rosen}) states that for any conjugacy class $C \subset G_n$, there is a constant $\Delta$ such that
\begin{equation} \label{Cheb}
\#\{P \in U : \text{$\deg P = k$ and $\text{Frob}(P) = C$}\} \leq \frac{\#C}{\#G_n}\cdot \frac{q^{mk}}{k} + \Delta \frac{q^{mk/2}}{k}.
\end{equation}
Both $U$ and $T$ contain all but finitely many places of $\F_{q^m}(t)$, and so there exists $k_1$ such that for any $k \geq k_1$, all places of degree $k$ lie in both $U$ and $T$. 
The set of $g \in G_n$ acting on $\{\alpha_1, \ldots, \alpha_{d^n}\}$ with at least one fixed point is a union of conjugacy classes of $G_n$, and it follows from \eqref{Frobeq} and \eqref{Cheb} that for $k \geq k_1$,
$$
\#\{P : \text{$\deg P = k$ and $\{\beta_1, \ldots, \beta_k\} \subset \phi^n(\F_{q^{mk}})$}\} \leq \FPP(G_n)\cdot \frac{q^{mk}}{k} + \Delta \frac{q^{mk/2}}{k}.
$$
Thus for $k \geq k_1$ we have
\begin{equation} \label{medstep}
\frac{\#\{\beta \in \phi^n(\F_{q^{mk}}) : \deg \beta = k\}}{k} \leq \FPP(G_n)\cdot \frac{q^{mk}}{k} + \Delta \frac{q^{mk/2}}{k}.
\end{equation}
From Lemma \ref{ffcount}, we have $\#\{\beta \in \phi^n(\F_{q^{mk}}) : \deg \beta < k\} \leq 2q^{mk/2}$, and \eqref{medstep} then gives
\begin{equation} \label{almostthere}
\#\phi^n(\F_{q^{mk}}) \leq \FPP(G_n)\cdot q^{mk} + (\Delta+2) q^{mk/2}
\end{equation}
for $k \geq k_1$. Finally, combining \eqref{almostthere} with Lemma \ref{perim} and equation 
\eqref{p1red}, we obtain for $k \geq k_1$, 
\begin{align*}
\frac{\#\Per(\phi, \PP^1(\F_{q^{mk}}))}{q^{mk}+1} & \leq \frac{\#\phi^n(\PP^1(\F_{q^{mk}}))}{q^{mk}+1} \leq \frac{\#\phi^n(\F_{q^{mk}})}{q^{mk}} + (q^{mk} + 1)^{-1} \\ & \leq \FPP(G_n) + (q^{mk} + 1)^{-1} + (\Delta+2) q^{-mk/2} 
\end{align*}
Let $\delta > 0$.  Taking $k_0$ large enough so that $k_0 \geq k_1$ and
$(q^{mk_0} + 1)^{-1} + (\Delta+2) q^{-mk_0/2} < \delta$
completes the proof. 
\end{proof}

We obtain the following Corollary of Theorem \ref{fppreduction}:

\begin{corollary} \label{fppreductioncor}
Let $\Fq$ be a finite field of characteristic $p$ and $\phi \in \Fq(x)$ have degree $d$ with $2 \leq d < p$. Then for every $\epsilon > 0$ there are positive integers $M$ and $k_0$ such that
\begin{equation*} 
\frac{\#\Per(\phi, \PP^1(\F_{q^{Mk}}))}{q^{Mk}+1} <
\FPP(\text{pgIMG}(\phi)/\Fq) + \epsilon
\end{equation*}
for all $k > k_0$. Moreover, $M \leq \limsup_{n \to \infty} m_n$, where $m_n = [(K_n^{\text{arith}} \cap \overline{\F_q}) : \F_q]$.
\end{corollary}

\begin{proof}
Let $\epsilon > 0$ be given. By definition $\FPP(\text{pgIMG}(\phi)/\Fq) = \lim_{i \to \infty} \FPP(G_i)$, and so there is an infinite set $I$ such that $\FPP(G_i) \leq \FPP(\text{pgIMG}(\phi)/\Fq) + \epsilon/2$ for any $i \in I$.  For each $i \in I$, we may take $\delta = \epsilon/2$ in Theorem \ref{fppreduction} to obtain $m_i$ and $k_0$ such that
\begin{equation*} 
\frac{\#\Per(\phi, \PP^1(\F_{q^{m_ik}}))}{q^{m_ik}+1} \leq \FPP(\text{pgIMG}(\phi)/\Fq) + \epsilon/2 + \epsilon/2
\end{equation*}
for all $k \geq k_0$. If $\limsup_{n \to \infty} m_n = \infty$, then any choice of $i \in I$ proves the Corollary. If $\limsup_{n \to \infty} m_n = L < \infty,$ then we may take $i \in I$ large enough so that $m_i \leq L$. 
\end{proof}

Recall that a finite extension $E$ of $\F_q(t)$ is geometric (over $\F_q(t))$ if $E \cap \overline{\F_q} = \F_q$. Hence $K_n^{\text{arith}}$ is geometric if and only if $m_n = 1$ for all $n \geq 1$.

\begin{corollary} \label{infsup}
Let $\Fq$ be a finite field of characteristic $p$ and $\phi \in \Fq(x)$ have degree $d$ with $2 \leq d < p$. Then 
\begin{equation*} 
\liminf_{k \to \infty}\frac{\#\Per(\phi, \PP^1(\F_{q^{k}}))}{q^{k}+1}  \leq \FPP(\text{pgIMG}(\phi)/\Fq).
\end{equation*}
If in addition $K_n^{\text{arith}}$ is geometric over $\F_q(t)$ for all $n \geq 1$, then
\begin{equation*} 
\limsup_{k \to \infty}\frac{\#\Per(\phi, \PP^1(\F_{q^{k}}))}{q^{k}+1}  \leq \FPP(\text{pgIMG}(\phi)/\Fq).
\end{equation*}
\end{corollary}

\begin{proof}
The first statement follows from Corollary \ref{fppreductioncor} and the second from Theorem \ref{fppreduction}.
\end{proof}

In particular, if $K_n^{\text{arith}}$ is geometric over $\F_q(t)$ for all $n \geq 1$ and $\FPP(\text{pgIMG}(\phi)/\Fq) = 0$, then the second statement of Corollary \ref{infsup} gives
\begin{equation*} 
   \lim_{k \to \infty}\frac{\#\Per(\phi, \PP^1(\F_{q^{k}}))}{q^{k}+1} = 0. 
\end{equation*}

At present the constant field sub-extensions $ \F_{q^{m_n}}$ (which we recall is  $K_n^{\text{arith}} \cap \overline{\Fq}$) are in general poorly understood. The main result is in the case of quadratic polynomials, and due to Pink:
\begin{theorem}[Pink \cite{pink2}] \label{pinkconst}
Let $\F_q$ be a finite field of odd characteristic, and let $\phi \in \F_q[x]$ have degree 2. Suppose that the unique finite critical point of $\phi$ is strictly preperiodic and that $\phi$ is not conjugate to a Chebyshev polynomial. Then
$$
K_n^{\text{arith}} \cap \overline{\Fq} \subseteq \F_q(\zeta_8),
$$
where $\zeta_8$ is a primitive $8$th root of unity. In particular, if $q$ is a square then $K_n^{\text{arith}}$ is geometric over $\F_q(t)$ for all $n \geq 1$.
\end{theorem}

Together with Corollary \ref{infsup}, this gives:
\begin{corollary} \label{limzerocor}
Let $\F_q$ be a finite field of odd characteristic, and let $\phi \in \F_q[x]$ have degree 2. Suppose that $q$ is a square, the unique finite critical point of $\phi$ is strictly preperiodic, and $\phi$ is not conjugate over $\overline{\F_q}$ to a Chebyshev polynomial. If $\FPP(\text{pgIMG}(\phi)/\Fq) = 0$, then 
\begin{equation*} 
   \lim_{k \to \infty}\frac{\#\Per(\phi, \PP^1(\F_{q^{k}}))}{q^{k}+1} = 0. 
\end{equation*}
\end{corollary}

We now wish to show that $\FPP(\text{pgIMG}(\phi)/\Fq) = \FPP(\text{pgIMG}(\tilde{\phi})/\C)$, thereby reducing the proofs of both Theorems \ref{main1} and \ref{main12} to the computation of $\FPP(\text{pgIMG}(\tilde{\phi})/\C)$.
To do so, we take advantage of theorems about lifting Galois groups from characteristic $p$ to characteristic $0$. Let $T$ and $T'$ be two complete $d$-ary rooted trees. If $\iota: T \to T'$ is an isomorphism of rooted trees, then any $G \leq \Aut(T)$ embeds as a subgroup $\iota \circ G \circ \iota^{-1}$ of $\Aut(T')$. A different choice of $\iota$ alters the image of this embedding by a conjugacy in $\Aut(T')$. In particular, $\FPP(G) = \FPP(\iota \circ G \circ \iota^{-1})$ independent of choice of $\iota$, since $\FPP$ is invariant under conjugacy. 

\begin{definition} \label{liftabilitydef}
Let $\Fq$ be a finite field of characteristic $p$ and let $\phi \in \Fq(x)$ have degree $d \geq 2$ with $\text{pgIMG}(\phi)/\Fq = G_\infty$ acting on the tree $T_{\Fq}(\phi)$ of preimages of $t$ in $\overline{\Fqbar(t)}$. We call $\phi$ \textbf{liftable} if there exists a map $\tilde{\phi} \in \C(x)$ with $\text{pgIMG}(\tilde{\phi})/\C = \tilde{G}_\infty$ acting on the tree $T_\C(\tilde{\phi})$ of preimages of $t$ in $\overline{\C(t)}$ such that 
\begin{enumerate}
\item $\phi$ and $\tilde{\phi}$ have the same ramification portrait, and  
\item there is a tree isomorphism $\iota: T_{\Fq}(\phi) \to T_\C(\tilde{\phi})$ such that $\iota \circ G_{\infty} \circ \iota^{-1} = \tilde{G}_\infty$. 
\end{enumerate}
\end{definition}

Not all $\phi \in \F_q(x)$ are liftable; for instance if $\phi(x) = x^p - x$ then $\infty$ is the only critical point in $\mathbb{P}^1(\overline{\F_q})$, and no lift $\tilde{\phi} \in \C(x)$ can have the same ramification portrait. 

Note that condition (2) of Definition \ref{liftabilitydef} ensures that if $\phi$ is liftable, then 
\begin{equation} \label{lifteqn}
    \FPP(\text{pgIMG}(\phi)/\Fq) = \FPP(\text{pgIMG}(\tilde{\phi})/\C).
\end{equation}
In Section \ref{background} we show that the latter is equal to $\FPP(\IMG(\tilde{\phi}))$ (see p. \pageref{imgequiv}).  


We remark that 
the action of $\text{pgIMG}(\tilde{\phi)}/\C$ on $T_\C(\tilde{\phi})$ is given by the action of the topological fundamental group $\pi_1(\mathbb{P}^1_\C\setminus P_{\tilde{\phi}}, z_0)$, where $z_0$ is any point outside of $P_{\tilde{\phi}}$ 
The latter may be computed by pulling back loops in 
$\mathbb{P}^1_\C \setminus P_{\tilde{\phi}}$, which allows for the use of topological and geometric tools.  





In order to harness these new tools, we need to know that the maps we study are liftable. For this we appeal to a result of R. Pink.
\begin{theorem}[Pink \cite{pink1}, Corollary 4.4]  \label{liftability}
Let $\Fq$ be a finite field of odd characteristic, and let $\phi \in \Fq(x)$ have degree 2. Then $\phi$ is liftable. 
\end{theorem}
 

To prove Theorem \ref{liftability}, Pink constructs a fine moduli scheme $M_\Gamma$ for $\Gamma$-marked quadratic morphisms, i.e. quadratic morphisms with specified ramification portrait $\Gamma$. The construction is explicit, and $M_\Gamma$ has several desirable properties, the most crucial being that it is quasi-finite over $\text{Spec} \, \Z[\frac{1}{2}]$ \cite[Theorem 3.3]{pink1}. These properties lead to a proof that any $\Gamma$-marked quadratic morphism over a finite field of odd characteristic $p$ lifts to characteristic zero: it is isomorphic to the special fiber of a $\Gamma$-marked quadratic morphism over $\text{Spec} \, R$, where $R$ is a discrete valuation ring that is finitely generated over $\Z_{(p)}$ \cite[Corollary 3.6]{pink1}. Liftability in the sense of Definition \ref{liftabilitydef} then follows as a direct consequence of Grothendieck's Specialization Theorem for tame fundamental groups; see \cite[Section 4]{pink}.  

The key step in Pink's argument is the quasi-finiteness of $M_\Gamma$, which is equivalent to the statement that that any quadratic morphism over a function field of characteristic $\neq 2$ is isotrivial, i.e. defined over a finite extension of the constant field after a change of variables. Using $p$-adic methods that are completely different from those of \cite{pink1}, this statement was proven in \cite[Corollary 6.3]{bijl}. 

Finally, note that because of condition (1) in Definition \ref{liftabilitydef}, a liftable map $\phi \in \Fq(x)$ is dynamically exceptional if and only if its lift is. From Theorem \ref{liftability}, Corollary \ref{infsup}, and Corollary \ref{limzerocor}, we then obtain: 
\begin{corollary} \label{finalliftcor}
Theorem \ref{maincx1} implies Theorem \ref{main1} and Theorem \ref{main12}.
\end{corollary}
 
We now turn to Question \ref{mainquestion2}, the ``horizontal" question involving finite fields of different characteristics. 
Recall that if $K$ is a number field and $\phi \in K(x)$, then for all but finitely many primes $\p$ in the ring of integers $\O_K$ of $K$, one may reduce the coefficients of $\phi$ modulo $\p$ to obtain a morphism $\phi_\p : \mathbb{P}^1(\mathbb{F}_\p) \to \mathbb{P}^1(\mathbb{F}_\p)$ with $\deg \phi = \deg \tilde{\phi}$, where $\mathbb{F}_\p$ is the residue field $\O_K/\p$. Denote by $N(\p)$ the degree of 
$\mathbb{F}_\p$ over its prime field, so that $1 + N(\p)$ is the size of $\mathbb{P}^1(\mathbb{F}_\p)$. 


\begin{theorem} \label{numfield}
Let $K$ be a number field and $\phi \in K(x)$. Then 
\begin{equation} \label{fpp2s}
\liminf_{N(\p) \to \infty} \frac{\#\Per(\phi_\p, \F_{\p})}{1 + N(\p)} \leq \FPP(\text{pgIMG}(\phi)/\C),
\end{equation}
where the lim inf is over primes $\p$ of $K$.
\end{theorem}

\begin{proof}
Let $K_n^{\text{geom}}$ be the splitting field of $\phi^n(x) - t$ over $\overline{\Q}(t)$, and $G_n = \Gal(K_n^{\text{geom}}/\overline{\Q}(t))$, so that $\lim_{n \to \infty} \FPP(G_n) = \FPP(\text{pgIMG}(\phi)/\overline{\Q})$. Because $K_n^{\text{geom}}$ is an algebraic extension of $\overline{\Q}(t)$, for any extension field $F$ of $\overline{\Q}$ we have that the field of constants of $K_n^{\text{geom}} \cap F(t)$ is an algebraic extension of $\overline{\Q}$. Hence 
$K_n^{\text{geom}} \cap F(t)= \overline{\Q}(t)$. By the theorem on natural irrationalities, it follows that the Galois group of the compositum $FK_n^{\text{geom}}$ over $F(t)$ is isomorphic to $G_n$. Choosing an embedding $\overline{\Q} \hookrightarrow \C$, we may take $F = \C$. This embedding can be extended to an embedding $\overline{\overline{\Q}(t)} \hookrightarrow \overline{\C(t)}$, which carries $T_{\overline{\Q}}(\phi)$ onto $T_\C(\phi)$. It follows that $\text{pgIMG}(\phi)/\overline{\Q} \cong \text{pgIMG}(\phi)/\C$, and the action of the former on $T_{\overline{\Q}}(\phi)$ is conjugate to the action of the latter on $T_\C(\phi)$ (where the conjugacy depends on the choice of embeddings). 

Let $L_n = K_n^{\text{arith}} \cap \overline{\Q}$, and recall that $K_n^{\text{arith}}\overline{\Q}(t) = K_n^{\text{arith}}\overline{\Q} = K_n^{\text{geom}}$. Then $K_n^{\text{arith}}\overline{\Q}(t)$ is a geometric extension of $L_n(t)$ with Galois group $G_n$, by the theorem on natural irrationalities.   
From \cite[Proposition 5.3]{juul1} we have that for primes $\mathfrak{P}$ of $L_n$ and for any $\delta > 0$,
\begin{equation} \label{fpp22}
 \frac{\#\Per(\phi_\mathfrak{P}, \F_{\mathfrak{P}})}{1 + N(\mathfrak{P})} \leq \FPP(G_n) + \delta
\end{equation}
for $N(\mathfrak{P})$ sufficiently large, where $N(\mathfrak{P})$ is the norm of $\mathfrak{P}$. From \cite[Lemma 6.3]{juul1} and \eqref{fpp22} we obtain 
\begin{equation} \label{fpp23}
\liminf_{N(\p) \to \infty} \frac{\#\Per(\phi_\mathfrak{p}, \F_{\mathfrak{p}})}{1 + N(\mathfrak{p})} \leq \FPP(G_n) + \delta,
\end{equation}
where $\p$ varies over primes of $K$. 
To prove \eqref{fpp2}, let $\epsilon > 0$. Let $n$ be such that $\FPP(G_n) \leq \FPP(\text{pgIMG}(\phi)/\overline{\Q}) + \epsilon/2$. Applying \eqref{fpp23} with $\delta = \epsilon/2$ gives
\begin{equation*} 
\liminf_{N(\p) \to \infty} \frac{\#\Per(\phi_\p, \F_{\p})}{1 + N(\p)} \leq \FPP(\text{pgIMG}(\phi)/\overline{\Q}) + \epsilon,
\end{equation*}
from which \eqref{fpp2s} follows, because $\FPP(\text{pgIMG}(\phi)/\overline{\Q}) = \FPP(\text{pgIMG}(\phi)/\C)$ by the first paragraph of the proof.  
\end{proof} 

Theorem \ref{numfield} shows that the only obstacle to proving Theorem \ref{main2} is establishing that $\FPP(\text{pgIMG}(\phi)/\C) = 0$. When $\phi$ is conjugate over $\overline{K}$ to a polynomial, this is Theorem 1.1 of \cite{fpfree}. If $\phi$ has prime degree or doubly transitive monodromy, this is Theorem \ref{maincx1}. We thus have:

\begin{corollary} \label{finalliftcor2}
Theorem \ref{maincx1} implies Theorem \ref{main2}.
\end{corollary}

\section{Background and definitions on IMGs and wreath recursion} \label{background}


The proof of Theorem \ref{maincx1}, which requires a proof of Theorem \ref{maincx2}, occupies the remainder of the article. 
From this section on, we work in a more topological context, and so use the notation $\rs$ in place of $\mathbb{P}^1_\C$. We now use $f$ to denote a rational function with complex coefficients, and we use $z$ as the variable. Given $f \in \C(z)$, we wish to understand the action of $\text{pgIMG}(f)/\C$ on $T_\C(f)$. In Section \ref{subs:wreath} we discuss tools for studying the action of an arbitrary group on a complete $d$-ary infinite rooted tree $X^*$. In Section \ref{subs:IMGs} we define the iterated monodromy group and describe its standard action on $X^*$. In Section \ref{subs:periphLoops} we give some basic properties of the monodromy action on roots of a polynomial that will be used in Section \ref{sec:IMGofPCF}.  

\subsection{Wreath recursion and definitions} \label{subs:wreath}
Let $d \geq 2$, put $X = \{0, \ldots, d-1\}$, and let $S_d$ denote the symmetric group on $d$ letters.  Denote by $X^*$ the set of all words in $X$, arranged as a tree in the natural way: there is an edge connecting $vx$ to $v$ for each $v \in X^*$ and $x \in X$. Denote by $X^n$ the set of words in $X$ of length $n$, which gives the $n$th level of $X^*$. By $X^0$ we mean the set consisting only of the empty word. An \textit{end} of $X^*$ is an infinite, non-retracing path beginning at the empty word. Thus the set of all ends of $X^*$ is the inverse limit of the $X^n$ under the natural maps $X^n \to X^{n-1}$. 

Define $\Aut(X^*)$ to be the set of tree automorphisms.
A salient feature of $X^*$ is its self-similarity, and we use this to describe elements of $\Aut(X^*)$ recursively.  

Let $g \in \Aut(X^*)$, and for a vertex $v \in X^*$ consider the subtrees $vX^*$ and $g(v)X^*$ with root $v$ and $g(v)$, respectively.  Both are naturally isomorphic to $X^*$, and identifying them gives an automorphism $g|_v \in \Aut(X^*)$, called the {\em restriction} of $g$ at $v$.

There is a natural isomorphism 
$$\psi : \Aut(X^*) \to S_d \wr \Aut(X^*),$$
where $\wr$ denotes the wreath product, that takes $g$ to $(\sigma, (g|_0, \ldots, g|_{d-1}))$, where $\sigma \in S_d$ is the action of $g$ on $X$ (i.e., on the first level of $X^*$).  In other words, we may describe $g$ by specifying its action on $X$ and its restriction at each element of $X$.  We call this the {\em wreath recursion} describing $g$.  We generally drop the outer parentheses and equate $g$ with its image under $\psi$, writing 
\begin{equation*}
g = \sigma(g|_0, \ldots, g|_{d-1}).
\end{equation*}  
We write the identity element as $1$, and when the permutation $\sigma$ is the identity, we omit it.  Hence the identity element of $\Aut(X^*)$ is given in wreath recursion by $(1,1, \ldots, 1)$.  Note that the element $a = (a, 1, 1, \ldots, 1)$ is also the identity, since by induction it acts trivially on $X^n$ for all $n$, and thus acts trivially on $X^*$.  Given $g = \sigma(g|_0, \ldots, g|_{d-1})$, we can make explicit its action on any $X^n$ thanks to the following formulas, which are straightforward to prove: 
\begin{equation} \label{nested restriction}
g|_{vw} = g|_v|_w \qquad g(vw) = g(v)g|_v(w),
\end{equation}
for any $v,w \in X^*$.  

One can multiply elements in wreath recursion form using the usual multiplication in a semi-direct product:
\begin{equation} \label{prod}
\sigma(g_0 \ldots,  g_{d-1}) \cdot \tau(h_0 \ldots,  h_{d-1}) = \sigma \tau(g_{\tau(0)}h_0 \ldots,  g_{\tau(d-1)}h_{d-1}).
\end{equation}
If we take $v \in X^*$ of length $n$, we may consider \eqref{prod} as giving the wreath recursion of $g, h \in \Aut(X^*)$ acting on $X^n$.  This gives 
\begin{equation} \label{singleelt}
(gh)(v) = g(h(v)) \qquad \text{and} \qquad (gh)|_v = g|_{h(v)} \cdot h|_v
\end{equation}

\begin{definition}
A subgroup $G$ of $\Aut(X^*)$ is \textit{level-transitive} if for all $n \geq 1$, $G$ acts transitively on $X^n$.
\end{definition}

\begin{definition} A subgroup $G$ of $\Aut(X^*)$ is \textit{self-similar} if for all $g \in G$ we have $g|_v \in G$ for every $v \in X^*$.
\end{definition}

\begin{definition}
A subgroup $G$ of $\Aut(X^*)$ is {\it recurrent} if $G$ is self-similar, $G$ acts transitively on $X$, and for each $x\in X$, the map 
\begin{equation} \label{virtual endo}
\text{$\{g \in G : g(x) = x\} \to G$ \; given by \; $g \mapsto g|_x$}
\end{equation}
is surjective. 
\end{definition}
We note that the map in \eqref{virtual endo} is known as the virtual endomorphism associated to $g$ and $x$. 

\begin{definition} \label{contractingdef}
A subgroup $G$ of $\Aut(X^*)$ is \textit{contracting} if $G$ is self-similar and there is a finite set $\mathcal{N} \subset G$ with the following property: for each $g \in G$, there is $M > 0$ such that $g|_v \in \mathcal{N}$ for every word $v \in X^*$ of length at least $M$.
\end{definition}

We record here a consequence of \cite[Corollary 2.8.5]{nek}:
\begin{proposition} \label{recprop}
A recurrent subgroup $G \leq \Aut(X^*)$ is level-transitive, and hence is infinite. 
\end{proposition}
\begin{proof}
The first assertion follows immediately from \cite[Corollary 2.8.5]{nek}. A level-transitive subgroup of $\Aut(X^*)$ must be infinite, because it acts transitively on arbitrarily large sets.
\end{proof}

\subsection{Basic properties of IMGs}\label{subs:IMGs}



Throughout this section, let $f: \rs \to \rs$ be a PCF rational function of degree $d\geq 2$ with post-critical set $P_f$ (the same construction works any expanding PCF branched cover $f : \sphere \to \sphere$ as in \cite{BD4}, but we will not use the extra generality here). Fix a choice of $z_0\in\rs\setminus P_f$.
Given $\gamma \in \pi_1(\rs\setminus P_f,z_0)$ and $z \in f^{-n}(z_0)$, there is a unique lift of $\gamma$ beginning at $z$, whose endpoint we denote $z_\gamma \in f^{-n}(z_0)$. The map $z\mapsto z_\gamma$ defines a permutation of $f^{-n}(z_0)$, and the resulting homomorphism 
$$
\pi_1(\rs\setminus P_f,z_0)  \to \text{Perm}(f^{-n}(z_0))
$$ 
is called the monodromy action of $\pi_1(\rs\setminus P_f,z_0)$ on $f^{-n}(z_0)$. Denote its kernel by $K_n$. The monodromy action extends to an action on the tree $T_{f, z_0} \subset \rs$ of preimages of $f$, rooted at $z_0$. (We use this notation rather than the previous $T_k(\phi)$ because this tree is a subset of $\rs$ rather than of $\overline{k(t)}$.) Its kernel is $K = \bigcap_{n=1}^\infty K_n$, which we call the \textit{faithful kernel} of the monodromy action.


\begin{definition} \label{imgdef}
With notation as above, the {\it iterated monodromy group} of $f$, written $\IMG(f)$, is the quotient of $\pi_1(\rs\setminus P_f,z_0)$ by the faithful kernel $K$ of the monodromy action on the tree $T_{f, z_0}$.  
\end{definition}

Select a labeling bijection $\Lambda : X \to f^{-1}(z_0)$, and for $i \in \{0, \ldots, d-1\}$ select a path $\ell_i$ from $z_0$ to $\Lambda(i)$ in $\rs\setminus P_f$. 
Then $\Lambda$ extends inductively to an isomorphism $\Lambda^* : X^* \to T_{f, z_0}$ of rooted trees via the rule
\begin{equation} \label{Lstar}
\Lambda^*(xv) = \text{end of the path $f^{-n}(\ell_x)$ starting at $\Lambda^*(v)$}
\end{equation}
for $v \in X^n$ \cite[Proposition 5.2.1]{nek}.

\begin{definition} \label{standardactiondef}
Fix choices of basepoint $z_0$, labeling map $\Lambda : X \to f^{-1}(z_0)$, and paths $\{\ell_i\}$. The corresponding standard action of $\pi_1(\rs\setminus P_f,z_0)$ (resp. $\IMG(f)$) on $X^*$ is the conjugation by $\Lambda^*$ of the monodromy action of $\pi_1(\rs\setminus P_f,z_0)$ (resp. $\IMG(f)$) on $T_{f, z_0}$.
\end{definition}

A standard action gives a homomorphism $\pi_1(\rs\setminus P_f,z_0) \to \Aut(X^*)$, which descends to an injective homomorphism $\IMG(f) \hookrightarrow \Aut(X^*)$ with identical image. Thus we may identify $\IMG(f)$ with a subgroup of $\Aut(X^*)$. A different choice of $z_0, \Lambda$, or $\{\ell_i\}$ only changes this group by a conjugacy in $\Aut(X^*)$. From now on we fix a standard action of $\pi_1(\rs\setminus P_f,z_0)$, and hence of $\IMG(f)$, on $\Aut(X^*)$. \label{fixed standard action}


 For given $n \geq 1$, it is a well-known result in the theory of Riemann surfaces that the permutation group of $f^{-n}(z_0)$ induced by the monodromy action of $\pi_1(\rs \setminus P_f, z_0)$ is identical (after possibly a conjugation in the symmetric group) to that given by the action of the Galois group $\Gal(\C(f^{-n}(t))/\C(t))$ on the set $f^{-n}(t) \subset \overline{\C(t)}$. Thus after possibly conjugating in $\Aut(X^*)$, we have that the action of $\text{pgIMG}(f)/\C$ on $f^{-n}(t)$ is the same as that of $\IMG(f)$ on $X^n$ (see e.g. \cite[Theorem 8.12]{forster}. Since $\text{pgIMG}(f)/\C$ is a closed subgroup of $\Aut(X^*)$ and it has the same image as $\IMG(f)$ under the restriction maps $\Aut(X^*) \to \Aut(X^n)$, it follows that $\text{pgIMG}(f)/\C$ is the closure of $\IMG(f)$ in $\Aut(X^*)$. This is \cite[Proposition 6.4.2]{nek}. In particular, we have \label{imginclusion}
$$\IMG(f) \subset \text{pgIMG}(f)/\C \subseteq \Aut(X^*)$$
and $\FPP(\text{pgIMG}(f)/\C) = \FPP(\IMG(f))$. \label{imgequiv}




We now describe a standard action in terms of wreath recursion. Equation \eqref{eqn:actionFormula} in the following proposition is found in Proposition 5.2.2 of \cite{nek}, and equation \eqref{restriction} is an immediate consequence of Definition \ref{standardactiondef} and equation \eqref{nested restriction}

\begin{proposition}\label{imgaction}
Given a standard action of $\pi_1(\rs\setminus P_f,z_0)$ (resp. $\IMG(f)$) on $X^*$, $\gamma \in \pi_1(\rs\setminus P_f,z_0)$ (resp. $\in \IMG(f))$, and $x \in X$, let $\tilde{\gamma}_x$ be the lift of $\gamma$ starting at $\Lambda(x)$. Then the action of $\gamma$ on $X^*$ is given by 
\begin{equation}\label{eqn:actionFormula}
\gamma(xv)=\gamma(x)(\ell_{\gamma(x)}^{-1} \tilde{\gamma}_x \ell_x)(v)
\end{equation}
where $\gamma(x)$ is the element of $X$ such that $\tilde{\gamma}_x$ ends in $\Lambda(\gamma(x))$. 
Moreover, for $v \in X^*$,

\begin{equation} \label{restriction}
\gamma|_{xv} =  [(\ell_{\gamma(x)})^{-1} \tilde{\gamma}_x \ell_x]|_v.
\end{equation}

\end{proposition}


A remark is in order about the statements in Proposition \ref{imgaction} regarding $\IMG(f)$. Because $\IMG(f)$ is a quotient of $\pi_1(\rs\setminus P_f,z_0)$, the quantities $\gamma$, $\tilde{\gamma}_x$, and $\gamma|_x$ are only defined up to elements of the faithful kernel. However, the elements of the faithful kernel act trivially on $T_{f,z_0}$, and hence do not affect the corresponding elements of $\Aut(X^*)$.



\begin{proposition}
\label{prop:recurrent}
A standard action of $\pi_1(\rs\setminus P_f,z_0)$ or $\IMG(f)$ on $\Aut(X^*)$ is recurrent. 
\end{proposition}

\begin{proof}
Let $G$ stand for either $\pi_1(\rs\setminus P_f,z_0)$ or $\IMG(f)$. Observe that Proposition \ref{imgaction} (with $v$ the empty word) implies that $G$ is self-similar. We now show that $G$ acts transitively on $X$. Let $i, j \in X$ and let $p$ be a path from $\Lambda(i)$ to $\Lambda(j)$ in $\rs\setminus f^{-1}(P_f)$. 
The path $f(p)$ has endpoints $f(\Lambda(i)) = f(\Lambda(j)) = z_0$, and thus $f(p)$ gives an element of $G$. 
Observe that the lift $\widetilde{f(p)}$ of $f(p)$ beginning at $\Lambda(i)$ is precisely $p$. By Proposition \ref{imgaction} we then have $(f(p))(i) = j$, showing that the action of $G$ on $X$ is transitive. 

Finally, we show that given $i \in X$ the virtual endomorphism $g\mapsto g|_{i}$ is a surjective map from $\{g \in G : g(i) = i\}$ to $G$.
Let $h\in G$ and take a representative curve for $h$ (which we will also refer to as $h$ in an abuse of notation) that avoids $f^{-1}(P_f)$.
Let the path $\bar{h}$ be the composition $\ell_i h\ell_i^{-1}$.
Notice that $\bar{h}$ is a loop in $\rs\setminus f^{-1}(P_f)$ based at $\Lambda(i)$.
So (the homotopy class of) $f(\bar{h})$ is a loop based at $z_0$, and thus gives an element of $G$. The lift of $f(\bar{h})$ beginning at $\Lambda(i)$ is $\bar{h}$, and thus $(f(\bar{h}))(i) = i$ by Proposition \ref{imgaction}. The same proposition then yields  
$$f(\bar{h})|_i = \ell_i^{-1}\bar{h}\ell_i = \ell_i^{-1}\ell_i h\ell_i^{-1} \ell_i,
$$ 
which is homotopic to $h$, and thus equals $h$ in $G$.
Therefore, the map $g\mapsto g|_{i}$ is onto.
\end{proof}



Proposition \ref{recprop} immediately gives:
\begin{corollary}
A standard action of $IMG(f)$ on $\Aut(X^*)$ is level-transitive, and hence $\IMG(f)$ is infinite.
\end{corollary}

To this point, the results of this section hold more generally for PCF branched self-covers of the sphere. However, if $f$ is specifically a post-critically-finite rational map, the expansion properties of $f$ have further implications for the iterated monodromy group. 
Let $P_f^{per}\subset P_f$ denote the union of all periodic orbits containing a critical point. By  \cite[Theorem 19.6]{MilnorBook}, $f$ is subhyperbolic \label{subhyperbolic} because every critical orbit is finite. That theorem is proved by constructing an orbifold metric on $\rs\setminus P_f^{per}$ so that for all $p\in \rs\setminus f^{-1}(P_f)$, the derivative satisfies 
\begin{equation}\label{eqn:contraction}
||Df(p)||>1.
\end{equation}
For $p\in P_f^{per}$ denote by $\mathcal{U}(p)$ an open B\"ottcher disk containing $p$ (as in \cite[Theorem 9.1]{MilnorBook}). There is a choice of the neighborhood $\mathcal{U}(p)$ for each $p\in P_f^{per}$ so that the collection \[\mathcal{U}^{per}:=\bigcup_{p\in P_f^{per}}\mathcal{U}_{p}\] has complement $K=\rs\setminus\mathcal{U}^{per}$ with the property that $K':=f^{-1}(K)$ is compactly contained in $K$. By compactness there is a constant $0<\rho<1$ so that 
\begin{equation}\label{derivBound}
||Df(p)||\geq \frac{1}{\rho} >1    
\end{equation}
 for all $p\in K'$. 
 

In the presence of this metric expansion, certain finiteness properties hold. For example, it was used by Nekrashevych to prove the following statement  on contraction (recall Definition \ref{contractingdef}) of self-similar groups  \cite[Theorem 5.5.3]{nek}.

\begin{theorem}
\label{cor:IMGnucleus}
If $f:\rs\to\rs$ is PCF, then $IMG(f)$ is contracting.
\end{theorem}

\subsection{Peripheral loops}\label{subs:periphLoops}
Let $f : \rs \to \rs$ be a PCF rational function, and recall that we have fixed a standard action of $\pi_1(\rs\setminus P_f,z_0)$ (and hence of $\IMG(f)$) on $X^*$. In Section \ref{sec:IMGofPCF} we study this action by analyzing loops, and here we record some elementary properties of loops that will prove useful. 

We say that a homotopy class of paths based at a point $z$ is a loop if it can be represented by a loop, or equivalently if every representative is a loop. 
The following lemma is an immediate consequence of Proposition \ref{imgaction}:

\begin{lemma}
\label{liftable}
The lift of $g\in\pi_1(\rs\setminus P_f,z_0)$ to $z \in T_{f, z_0}$ is a loop if and only if $g(\Lambda^*(z))=\Lambda^*(z)$.
\end{lemma}

\begin{definition}
A nontrivial element $g\in\pi_1(\rs\setminus P_f,z_0)$ is \emph{peripheral about $p\in P_f$} if for any disk neighborhood $N(p)$ of $p$ there exists a representative of $g$ that is freely homotopic (i.e. homotopic with continuously moving basepoint) in $\rs \setminus P_f$ to a loop that is contained in $N(p)$. We call $g$ \textit{peripheral} if there exists a $p\in P_f$ so that $g$ is peripheral about $p$. 
\end{definition}

\begin{definition}
A nontrivial element $g\in\pi_1(\rs\setminus P_f,z_0)$ is
called \emph{primitive} if $g=h^m$ for $h\in \pi_1(\rs\setminus P_f,z_0)$ implies that $m=1$ or $m=-1$. 
\end{definition}






Fix a disk neighborhood $N(p)$ of $p$ so that each component of $f^{-1}(N(p))$ contains at most one element of $f^{-1}(p)$. Let $g\in\pi_1(\rs\setminus P_f,z_0)$ be peripheral about $p$, which by definition means that there is a loop $g_p$ that is freely homotopic to $g$ and contained in $N(p)$.  A lift $\widetilde{g}$ of $g$ is said to be \emph{associated to a point} $q\in f^{-1}(p)$ if the free homotopy $g \simeq g_p$ lifts to a free homotopy $\tilde{g} \simeq \tilde{g}_q$ (in $\rs \setminus f^{-1}(P_f)$) where $\tilde{g}_q$ is contained in the component of $f^{-1}(N_p)$ that contains $q$. We note that given $z \in f^{-1}(z_0)$ and $g$ peripheral about $p \in P_f$, the lift of $g$ beginning at $z$ is associated to precisely one $q \in f^{-1}(p)$.


\begin{lemma}
\label{perilift}
Let $g\in\pi_1(\rs\setminus P_f,z_0)$ be primitive and peripheral about $p \in P_f$, and let $\tilde{g}$ be a lift of $g$ beginning at $z \in f^{-1}(z_0)$. Suppose that $\tilde{g}$ is associated to $q\in f^{-1}(p)$. Then $q$ is non-critical if and only if $g(\Lambda(z))=\Lambda(z)$. 
\end{lemma}

\begin{proof}
Let $\tilde{g}$ be a lift of $g$ associated to $q$, and let $U(q)$ be the component of $f^{-1}(N(p))$ that contains $q$. By Lemma \ref{liftable} we have $g(\Lambda(z)) = \Lambda(z)$ if and only if $\tilde{g}$ is a loop. By definition $\tilde{g}$ is freely homotopic to $\widetilde{g_q} \subset U(q)$ that is a lift of a loop $g_p \subset N(p)$ freely homotopic to $g$. It follows from the homotopy lifting property that $\tilde{g}$ is a loop if and only if $\widetilde{g_q}$ is a loop. 

Because $f$ is a branched cover, the restriction $f : U(q) \to N(p)$ is modeled on the unit disk map $z\mapsto z^d$ where $d \geq 1$ is the local degree of $f$ at $q$.  A primitive nontrivial loop in $\mathbb{D}\setminus\{0\}$ lifts to a loop under $z\mapsto z^d$ if and only if $d=1$, i.e. if and only if $q$ is non-critical. 
\end{proof}



\begin{lemma}
\label{periliftable}
Let $g\in\pi_1(\rs\setminus P_f,z_0)$ be primitive and peripheral about $p \in P_f$, and let $\tilde{g}$ be a lift of $g$ beginning at $z \in f^{-1}(z_0)$. If $\tilde{g}$ is a loop, then it is either trivial in $\pi_1(\rs\setminus P_f,z_0)$ or it is peripheral about a non-critical point in $P_f$.
\end{lemma}

\begin{proof}
Let $q \in f^{-1}(p)$ be such that $\tilde{g}$ is associated to $q$, let $N(q)$ be a disk neighborhood of $q$, and assume that $\tilde{g}$ is a loop. Because $g$ is peripheral about $p$, we can select a loop $g_p$ that is freely homotopic to $g$ and contained in a neighborhood $N(p)$ of $p$ such that $f^{-1}(N(p))$ has a component contained in $N(q)$. Then $\tilde{g}_q$ is freely homotopic to a loop contained in $N(q)$, and hence is peripheral about $q$. Note that if $q \not\in P_f$, then $\tilde{g}$ is trivial in $\pi_1(\rs\setminus P_f,z_0)$. Because $\tilde{g}$ is a loop we have from Lemma \ref{liftable} that $g(\Lambda(z)) = \Lambda(z)$. Hence by Lemma \ref{perilift} we have that $q$ is non-critical. 
\end{proof} 

The following lemma connects the dynamical properties of the post-critical set to to the action of a loop on the tree of preimages.

\begin{lemma}
\label{periend}
Let $g\in\pi_1(\rs\setminus P_f,z_0)$ be primitive and peripheral about $p \in P_f$. Then
\begin{enumerate}
\item $g$ fixes an end of $X^*$ 
if and only if 
there is a backward orbit of $p$ that does not contain a critical point, and
\item $g$ fixes infinitely-many ends of $X^*$ if there is a backward orbit of $p$ that contains no critical point and is not a subset of $P_f$.
\end{enumerate} 
\end{lemma}
\begin{proof}
The first statement follows from Lemma \ref{perilift}.
The second statement follows from the fact that the trivial action on a subtree fixes all ends of that subtree.
\end{proof}

\section{The fixed-point process for self-similar groups} \label{fpprocess}

Throughout this section, we assume $X = \{0, \ldots, d-1\}$ for $d \geq 2$, and let $X^n$ be the collection of words in $X$ of length $n$. In particular, $X = X^1$. Recall that we have fixed a standard action of $\pi_1(\rs\setminus P_f,z_0)$ (and hence of $\IMG(f)$) on $X^*$.  As in \eqref{n1def} in the introduction, we put 
\begin{align*} 
\mathcal{N}_1 & = \{g \in \img(f) : \text{$g|_v = g$ and $g(v) = v$ for some non-empty $v \in X^*$}\}.
\end{align*}
We denote by $\mathcal{N}_1(G)$ the analogous set for an arbitrary $G \leq \Aut(X^*)$ \label{ndefs}

In this section, we prove the following result, which is a key step in the proof of Theorem \ref{maincx1}: 

\begin{theorem} \label{cxthm1}
Let $f \in \C(z)$ be a PCF rational function of degree $d \geq 2$. Assume that $d$ is prime or that $f$ has doubly transitive monodromy. If every $g \in \mathcal{N}_1$ fixes infinitely many ends of $X^*$, then $\FPP(\IMG(f)) = 0$.
\end{theorem}

For each $n \geq 1$, let $G_n$ denote the quotient of $G$ by the kernel of the restriction map $G \to \Aut(X^n)$. Recall that the profinite completion $G_\infty$ of $G$ with respect to the $G_n$ (equivalently, the inverse limit of $G_n$ under the restriction maps $G_n \to G_{n-1}$) is a compact group, and its normalized Haar measure is a probability measure $\mu$ that projects to the discrete uniform measure on each $G_n$.  Moreover, $G_\infty$ carries a natural action on the set of ends $X^{\omega}$. The key step in the proof of Theorem \ref{cxthm1} is the following result.

\begin{theorem}  
\label{evconst} Suppose that $G \leq \Aut(X^*)$ is self-similar and level-transitive. If either 
\begin{enumerate}
    \item $d$ is prime, or
    \item $G$ is recurrent and acts doubly transitively on $X$,
\end{enumerate}
then
$$\mu(\{g \in G_\infty : \text{$g$ fixes infinitely many elements of $X^\omega$}\}) = 0.$$
\end{theorem}

Recall that $G$ acts doubly transitively on $X$ if for all $i, j, k, \ell \in X$ with $i \neq j$ and $k \neq \ell$, there exists $g \in G$ with $g(i) = k$ and $g(j) = \ell$/

Theorem \ref{evconst} is proven in Corollaries \ref{evconstcor}, \ref{fpprocessmain}, and \ref{dtranscor}. The same conclusion as in Theorem \ref{evconst} is reached in Theorem 1.4 of \cite{fpfree} under the assumption that $G$ contains a spherically transitive element, which implies that $G$ is level-transitive, though not necessarily self-similar. We remark too that in the special case $d = 2$, Theorem 1.2 of \cite{galmart} implies the conclusion of Theorem \ref{evconst} under the assumptions that $G$ is level-transitive and for each $n$ the sign homomorphism $\text{sgn}_n : G_n \to \{\pm 1\}$ is surjective. In this paper we must handle groups with $d = 2$ that do not have a spherically transitive element, and for which $\text{sgn}_n$ has trivial image for all $n$ sufficiently large.

Here is a sketch of the proof of Theorem \ref{cxthm1}. We define a stochastic process -- that is, an infinite collection of random variables defined on a common probability space -- that encodes information about the number of fixed points in $X^n$ of elements of $G_n$.  We then generalize the techniques of \cite{galmart} and \cite{fpfree} and  to show that this process is a martingale provided only that $G$ is self-similar and level-transitive. An application of a martingale convergence theorem and a result of Nekrashevych on contracting actions of iterated monodromy groups yield the final steps in the proof of Theorem \ref{cxthm1}.


We now give the precise construction and proofs.

Let a group $G$ act on a set $S$, and for $g \in G$ put
$\Fix(g) = \{s \in S : g(s) = s\}.$
Define a stochastic process $Y_1, Y_2, \ldots$ on $G_\infty$ by taking 
$$Y_i(g) = \#\Fix(\pi_i(g)),$$ 
where $\pi_i$ is the restriction map $G_\infty \to G_i$, and $G_i$ acts on $X^i$.  We call this the {\em fixed point process} of $G$.
Because $\mu(\pi_i^{-1}(T)) = \#T/\#G_i$ for any $T \subseteq G_i$, we have that $\mu(Y_1 = t_1, \ldots, Y_n = t_n)$ is given by 
\begin{equation} \label{fpchar}
\frac{1}{\# G_n}\# \left\{g \in G_n : \mbox{$g$ fixes $t_i$ elements of $X^i$ for 
$i = 1,2, \ldots, n$} \right\}.
\end{equation}
We denote by $E(Y)$ the expected value of the random variable $Y$.

\begin{definition}
A stochastic process with probability measure $\mu$ and random variables $Y_1, Y_2, \ldots$ taking values in $\mathbb{R}$ is a {\em martingale} if for all $n \geq 2$ and any $t_i \in \mathbb{R}$, 
$$E(Y_n \mid Y_{1} = t_{1}, Y_2 = t_2, \ldots, Y_{n-1} = t_{n-1}) = t_{n-1},$$
provided $\mu(Y_{1} = t_{1}, Y_2 = t_2, \ldots, Y_{n-1} = t_{n-1}) > 0$.  
\end{definition}

Martingales are useful tools because they often converge in the following sense:
\begin{definition}
Let $Y_1, Y_2, \ldots$ be a stochastic process defined on the probability space $\Omega$ with probability measure $\mu$.  
The process {\em converges} if 
$$\mu \left(\omega \in \Omega : \text{$\inflim{n} Y_n(\omega)$ exists} \right) = 1.$$
\end{definition}
We give one standard martingale convergence theorem (see e.g. \cite[Section 12.3]{grimmett} for a proof).
\begin{theorem} \label{martconv}
Let $M = (Y_1, Y_2, \ldots)$ be a martingale whose random variables take nonnegative real values.  Then $M$ converges.
\end{theorem}
Since the random variables in the fixed-point process take nonnegative integer values, we immediately have the following:  
\begin{corollary} \label{evconstcor}
Let $G \leq \Aut(X^*)$ and suppose that the fixed-point process for $G$ is a martingale.  Then 
$$\mu(\{g \in G_\infty : \text{$Y_1(g), Y_2(g), \ldots$ is eventually constant}\}) = 1.$$
In particular, 
$$\mu(\{g \in G_\infty : \text{$g$ fixes infinitely many elements of $X^\omega$}\}) = 0.$$
\end{corollary}

Thus to prove Theorem \ref{evconst}, it suffices to show that the fixed-point process for $G$ is a martingale. We therefore characterize when this happens. Let $H_n$ be the kernel of the restriction map $G_n \to G_{n-1}$.

\begin{theorem} \label{martchar}
Let $G \leq \Aut(X^*)$. Then the fixed-point process for $G$ is a martingale if and only if for all $n \geq 1$ and $v \in X^{n-1}$, $H_n$ acts transitively on the set $v* = \{vx : x \in X\}$.
\end{theorem}

\begin{proof}
Assume that $H_n$ acts transitively on $v*$. We must show
\begin{equation} \label{sloop}
E(Y_{n} \mid Y_1 = t_1, \ldots, Y_{n-1} = t_{n-1}) = t_{n-1},
\end{equation}
where $t_1, \ldots, t_{n-1}$ satisfy $\mu (Y_1 = t_1, \ldots, Y_{n-1} = t_{n-1}) > 0.$  Because the $Y_i$ take integer values, each $t_i$ must be an integer.  
By definition, the left-hand side of \eqref{sloop} is
\begin{equation} \label{sloop1}
\sum_k k \cdot \frac{\mu (Y_1 = t_1, \ldots, Y_{n-1} = t_{n-1}, Y_n = k )}
{\mu (Y_1 = t_1, \ldots, Y_{n-1} = t_{n-1})}.
\end{equation}
Put 
\begin{eqnarray*}
S & = & \{g \in G_n : \text{$g$ fixes $t_i$ elements of $X^i$ for $1 \leq i \leq n-1$}\} \\
S_k & = & \{g \in S : \text{$g$ fixes $k$ elements of $X^n$}\} 
\end{eqnarray*}
By \eqref{fpchar}, the expression in \eqref{sloop1} is equal to
$\sum_k k \cdot (\# S_k/\# S)$.
This in turn may be rewritten
\begin{equation} \label{sloop4}
\frac{1}{\#S} \sum_{g \in S} \#\Fix(g).
\end{equation}

Each $H_n$ acts trivially on $X^{n-1}$, so $S$ is invariant under multiplication by elements of $H_n$, whence $S$ is a union of cosets of $H_n$.  
Take $gH_n \subseteq S$, and let 
$$R = \{vx : v \in X^{n-1}, g(v) = v, x \in X\}.$$
Note that because $g \in S$, we have $\#R = dt_{n-1}$.  If $vx \in R$, then $g(vy) = vx$ for some unique $y \in X$.  Because $H_n$ acts transitively on $v*$, the set
$$
Q := \{h \in H_n : h(vx) = vy\}
$$
is non-empty, and is thus a coset of $\Stab_{H_n}(vx)$. By standard group theory, we then have $\#Q = \#H_n/\#O_{H_n}(vx) = \#H_n/d$, where the last equality follows from the transitivity of the action of $H_n$ on $v*$.

Now let $I(g,s)$ be the function that takes the value $1$ when $g(s) = s$ and $0$ otherwise. Then we have $\sum_{h \in H_n} I(gh, vx) = \#Q$ and hence
\begin{equation*}
\sum_{vx \in R}\sum_{h \in H_n} I(gh, vx) = \#Q \cdot dt_{n-1} = \#H_n t_{n-1}.
\end{equation*}
Inverting the order of summation and using that $g(w) \neq w$ for $w \not\in R$, we have
$$
\sum_{h \in H_n} \# \Fix(gh) = \#H_{n} t_{n-1}.
$$
But $S$ is a disjoint union of cosets of $H_n$, and hence the expression in \eqref{sloop4} equals $t_{n-1}$.

Assume now that $H_n$ does not act transitively on $v*$ for some $v \in X^{n-1}$. Then the action of $H_n$ on $X^n$ has $k$ orbits for some $k > d^{n-1}$, and so by Burnside's lemma we have 
$$\frac{1}{\#H_n} \sum_{h \in H_n} \# \Fix(h) = k > d^{n-1}.$$ Because $H_n$ is the full set of elements of $G_n$ that fix all $d^i$ elements of $X^i$ for each $i = 1, \ldots, d-1$, we have
$$E(Y_{n} \mid Y_1 = d, \ldots, Y_{n-1} = d^{n-1}) = k > d^{n-1},$$
and hence the fixed-point process for $G$ is not a martingale. 
\end{proof}

\begin{remark}
When $G$ has a spherically transitive element, it is straightforward to see that $H_n$ acts transitively on each set $v*$; indeed, a suitable power of the spherically transitive element will give such a transitive action. This together with Theorem \ref{martchar} gives a proof of \cite[Theorem 4.2]{fpfree}. 
\end{remark}

In light of Theorem \ref{martchar}, we examine the action of $H_n$ on $X^n$. 
\begin{lemma} \label{Horb}
Let $G \leq \Aut(X^*)$ act transitively on $X^n$. Let $H_n$ be the kernel of the restriction $G_n \to G_{n-1}$. Then the action of $H_n$ on $X^n$ consists of orbits of equal length $r$ for some $r \mid d$. 
\end{lemma}

\begin{proof}
Let $u, w \in X^n$. By the transitivity of the action of $G$ on $X^n$, there is $g \in G_n$ with $g(u) = w$. If $h(u) = u'$ for $h \in H_n$, then $h^g(w) = g(u')$, where $h^g := ghg^{-1} \in H_n$. Thus $g$ furnishes a map $O_{H_n}(u) \to O_{H_n}(w)$, which is invertible since $g$ is a permutation of $X^n$. Hence $\#O_{H_n}(u) = \#O_{H_n}(w)$. Now for any $v \in X^{n-1}$, $H_n$ preserves $v* = \{vx : x \in X\}$. Thus $v*$ is a set of $d$ elements that is a disjoint union of $H_n$-orbits. It follows that each orbit of $H_n$ has $r$ elements for $r \mid d$. 
\end{proof}

\begin{corollary} \label{transcor}
Let $d$ be prime and $G \leq \Aut(X^*)$. Suppose that $G$ is level-transitive and $H_n$ is non-trivial for all $n \geq 1$. Then for all $n \geq 1$ and all $v \in X^{n-1}$, $H_n$ acts transitively on the set $v*$.
\end{corollary}

\begin{proof}
We may apply Lemma \ref{Horb} thanks to the level-transitivity of $G$, and the non-triviality of $H_n$ gives $r > 1$. But $d$ is prime, and so $r = d$. Now each orbit of $H_n$ is contained in $v*$ for some $v \in X^{n-1}$, and thus each orbit equals $v*$ for some $v \in X^{n-1}$.  
\end{proof}

\begin{remark}
When $d = 2$, there is in fact a \textit{single element} of $H_n$ that acts transitively on $v*$ for all $v \in X^{n-1}$. Indeed, in this case $G_n$ is a 2-group, and so by the class equation every non-trivial normal subgroup of $G_n$ has non-trivial intersection with the center $Z(G_n)$ of $G_n$. Hence there is non-trivial $h \in H_n \cap Z(G_n)$. If $h(w) = w$ for some $w \in X^n$ then $h^g(g(w)) = g(w)$ for any $g \in X^n$, and thus $h(g(w)) = g(w)$. The transitivity of $G_n$ then gives $h = e$, a contradiction. Thus $h$ acts without fixed points on $X^n$, and since $d = 2$ this is equivalent to $h$ acting transitively on each $v*$. 
\end{remark}

In light of Corollary \ref{transcor}, in some sense the crucial question is to determine when $H_n$ is nontrivial for all $n \geq 1$. When $d = 2$, it is shown in \cite[Corollary 4.9]{galmart} that when $\text{sgn}_n$ is surjective for all $n \geq 1$, then $H_n$ is non-trivial for all $n \geq 1$, but the proof is quite involved. Here, in contrast to \cite[Corollary 4.9]{galmart}, we assume that $G$ is self-similar, and this allows for a much simpler proof of a much more general result.  

To streamline our argument, we define a function $v : G \to \Z_{\geq 0} \cup \{\infty\}$ by $v(e) = \infty$ for the identity $e \in \Aut(X^*)$ and 
$$
v(g) = \max \{n \geq 0 : \text{$g$ acts trivially on $X^n$}\}
$$
for $e \neq g \in \Aut(X^*)$. 
Note that each $g \in \Aut(X^*)$ fixes the lone element of $X^0$, and hence $v(g) \geq 0$. Moreover, for $n \geq 1$, $H_n$ is non-trivial if and only if $n \in v(G)$. Finally, we remark that $v(g) = n \geq 1$ if and only if $g$ acts trivially on $X^1$ and 
\begin{equation} \label{valrest}
\min \{ v(g|_x) : x \in X  \} = n-1
\end{equation}


\begin{proposition} \label{valsurj}
Let $G \leq \Aut(X^*)$ be infinite and self-similar. Then $v$ is surjective. 
\end{proposition}

\begin{proof} Suppose first that there is $N \geq 0$ with $v(g) \leq N$ for all $g \in G \setminus \{e\}$. We claim that the natural quotient map $\pi_N : G \twoheadrightarrow G_N$ is an isomorphism, and thus $G$ is finite. Indeed, if $\pi_N(g) = \pi_N(h)$, then $gh^{-1}$ acts trivially on $X^N$, and hence $v(gh^{-1}) > N$. Thus $gh^{-1} = e$, proving the claim.  

Therefore the infinitude of $G$ implies that $v(G)$ is infinite. Suppose now that $n \in v(G)$ for some $n \geq 1$, and let $g \in G$ with $v(g) = n$. From \eqref{valrest} there is $x \in X$ with $v(g|_x) = n-1$. By the self-similarity of $G$, we have $g|_x \in G$, and thus $n-1 \in v(G)$. By induction $\{0,1, \ldots n\} \subseteq v(G)$. The infinitude of $v(G)$ then implies that $v$ is surjective. 
\end{proof}

We remark that Proposition \ref{valsurj} is not true in general if $G$ fails to be self-similar. For example, let $d = 2$ and consider the group $J = \{e, (00 \; 11)(01 \; 10)\} \leq \Aut(X^2)$. Then the iterated wreath product of $J$ gives a closed subgroup $G \leq \Aut(X^*)$ with the property that $2n \not\in v(G)$ for all $n \geq 1$. Note that in this case $G$ is a self-similar subgroup of $\Aut(Y^*)$, where $Y = X^2$.

\begin{corollary} \label{fpprocessmain}
Let $d$ be prime and $G \leq \Aut(X^*)$. Suppose that $G$ is self-similar and level-transitive. Then the fixed-point process associated to $G$ is a martingale.
\end{corollary}

\begin{proof}
The level-transitivity of $G$ implies that $G$ is infinite, and the Corollary then follows from Theorem \ref{martchar}, Corollary \ref{transcor}, and Proposition \ref{valsurj}.
\end{proof}

\begin{theorem}
\label{thm:dt}
Let $G \leq \Aut(X^*)$ be a recurrent group whose action on $X$ is doubly transitive. Then for all $w\in X^n$ and  $i,j\in X$ with $i \neq j$, there exists $g\in H_n$ such that $g|_w$ takes $i$ to $j$.
\end{theorem}

\begin{proof}
First note that by Proposition \ref{recprop}, $G$ is infinite. By Proposition \ref{valsurj}, the function $v:G\to\Z\cup\{\infty\}$ defined by
$$v(g) = \max\{n\geq0\mid g \text{ acts trivially on }  X^n\}$$
is surjective, so there exists $g \in G$ with $v(g) = n$, i.e. $g\in H_n$ and $g$ is non-trivial. 

By Proposition \ref{recprop}, $G$ is level-transitive. Thus by passing to a conjugate we may assume that $g$ acts non-trivially on $w\ast =\{wx\mid x\in X\}$.
Let $h=g|_w$.
Since $h$ acts non-trivially on $X$, there exist $k, \ell \in X$ with $k \neq \ell$ such that $h(k)=\ell$.

By double-transitivity, we can choose $t\in G$ such that $t(i) = k$ and $t(j) = \ell$.
Since the action of $G$ is recurrent, we can choose $s\in G$ such that $s(w) = w$ and $s|_w = t$. Now $s^{-1}gs$ fixes $w$ and is also in $H_n$, because $H_n$ is a normal subgroup of $\Aut(X^*)$. 
From \eqref{singleelt} we then have 
$$(s^{-1}gs)|_w = s^{-1}|_wg|_ws|_w=t^{-1}ht.$$
But $(t^{-1}ht)(i) = j$, as desired.
\end{proof}

Theorems \ref{martchar} and \ref{thm:dt} immediately give:
\begin{corollary} \label{dtranscor}
Let $G \leq \Aut(X^*)$ be a recurrent group whose action on $X$ is doubly transitive. Then the fixed-point process for $G$ is a martingale.
\end{corollary}

\begin{proof}
By Theorem \ref{thm:dt}, for all $n\geq 1$, and all $v\in X^{n-1}$, the action of the elements of $G$ that act trivially on $X^{n-1}$ is transitive on the set $v\ast=\{vx\mid x\in X\}$.
Notice that the images under the quotient map to $G_n$ of elements of $G$ that act trivially on $X^{n-1}$ lie in $H_n$.
Thus, by Theorem \ref{martchar}, the fixed-point process for $G$ is a martingale.
\end{proof}

Suppose now that $G$ is contracting, and let $\mathcal{N} \subset G$ be a finite set as in Definition \ref{contractingdef}. If $g \in \mathcal{N}_1(G)$, then by definition there is $v \in X^*$ with $g(v) = v$ and $g|_v = g$, and hence taking $w_n$ to be the concatenation of $v$ with itself $n$ times, we have $g|_{w_n} = g$. It follows that $g \in \mathcal{N}$, and hence $\mathcal{N}_1(G)$ is finite. 







We now provide the final step in the proof of Theorem \ref{cxthm1}.
  
\begin{theorem} \label{crystal}
Suppose that $G \leq \Aut(X^*)$ is contracting and its fixed point process is a martingale. 
If every $g \in \mathcal{N}_1(G)$ fixes infinitely many ends of $X^*$, then $\FPP(G) = 0$.  
\end{theorem}

\begin{proof} This is proven in \cite[p. 2033]{fpfree}, but we give the argument here for completeness. Let $\mathcal{N} \subset G$ be a finite set as in Definition \ref{contractingdef}. Suppose that $g \in G$ fixes some end $w = x_1x_2\cdots$ of $X^*$.  Let $v_n = x_1x_2 \cdots x_n$ for each $n \geq 1$, and consider the sequence of restrictions $g|_{v_1}, g|_{v_2}, \ldots$.  For $n$ large enough, we have $g|_{v_n} \in \mathcal{N}$, and $g|_{v_n}$ fixes the end $x_{n+1}x_{n+2} \cdots$ since $g$ fixes $w$.  Because $\mathcal{N}$ is finite, there must be $i < j$ with $g|_{v_i} = g|_{v_j}$.  Let $h = g|_{v_i}$, and note that for $w = x_{i+1}x_{i+2} \cdots x_j$ we have $h(w) = w$ and $h|_w = h$.  Hence $h \in \mathcal{N}_1(G)$, and by hypothesis fixes infinitely many ends of $X^*$.   Inserting $v_i$ on the beginning of each of these ends, we obtain infinitely many ends of $X^*$ fixed by $g$.  Hence by Corollary \ref{evconstcor}, $g$ lies in a set of measure zero, proving the theorem.  
\end{proof}

\begin{proof}[Proof of Theorem \ref{cxthm1}]
This is an immediate consequence of Corollary \ref{fpprocessmain}, Corollary \ref{dtranscor}, Theorem \ref{crystal}, and the fact that any standard action of $\IMG(\psi)$ on $X^*$ is recurrent and contracting by Proposition \ref{prop:recurrent} and Corollary \ref{cor:IMGnucleus}
\end{proof}




\section{Iterated monodromy action of PCF rational functions}\label{sec:IMGofPCF}

In light of Theorem \ref{cxthm1}, the proof of Theorem \ref{maincx1} will be complete once we establish Theorem \ref{maincx2}, which we restate here for the convenience of the reader. 
First recall that we have fixed a standard action of $\pi_1(\rs\setminus P_f,z_0)$ (and hence of $\IMG(f)$) on $X^*$, and recall the definition of $\mathcal{N}_1$ from \eqref{n1def} (or the beginning of Section \ref{fpprocess}).


\begin{theorem}[Theorem \ref{maincx2}] \label{repeat_maincx2}
Let $f \in \C(z)$ be a PCF rational function that is not dynamically exceptional. Then every element of $\mathcal{N}_1$ fixes infinitely many ends of $X^*$. 
\end{theorem}

The key dynamical property underlying the proof of Theorem \ref{repeat_maincx2} is subhyperbolicity, i.e. that PCF rational functions are expanding away from periodic post-critical points in the orbifold metric as described on p. \pageref{subhyperbolic}. We observe that this expansion fails to hold in general for PCF branched covers $f : \sphere \to \sphere$, and there exist such covers (necessarily not rational functions) that are not dynamically exceptional yet have elements of $\mathcal{N}_1$ fixing only finitely many ends of $X^*$.

The converse of Theorem \ref{repeat_maincx2} holds as well, thus giving a characterization of exceptional rational functions. Though it is not necessary for this paper, we give a proof in Theorem \ref{thm:ExceptionalChar}.

\subsection{End behavior of non-exceptional maps: fundamental group}
The proof of Theorem \ref{repeat_maincx2} relies on lifts of loops representing elements of $\IMG(f)$. We thus work first on the level of the fundamental group and later argue that nothing is lost when passing to the faithful quotient (Proposition \ref{modok}). We define the fundamental group version of $\mathcal{N}_1$, noting that it depends on the choice of standard action made on p. \pageref{fixed standard action}.
\begin{equation}\label{eqn:piN1}
\mathcal{N}_1^{\pi} := \{g\in \pi_1(\rs\setminus P_f) :\ \exists \text{ non-empty } w\in X^* \text{ so that } g(w) = w \text{ and }g|_w = g \}.
\end{equation}

 The basepoint of the fundamental group in \eqref{eqn:piN1} is not specified because the definition is independent of basepoint in the following narrow sense. 
 Let $\alpha$ be a path in $\rs\setminus P_f$ that connects a new basepoint $z_1$ to the original basepoint $z_0$. The map $\alpha_*:\pi_1(\rs\setminus P_f,z_0)\to \pi_1(\rs\setminus P_f,z_1)$ defined by $\alpha_*(g)=\alpha^{-1}g\alpha:=g^{\alpha}$ is an isomorphism. We define a standard action of $\pi_1(\rs\setminus P_f,z_1)$ on $X^*$ by taking the paths connecting $z_1$ to $f^{-1}(z_1)$ to be $\tilde{\alpha}_x^{-1}\ell_x\alpha$ where $\tilde{\alpha}_x$ is the unique lift of $\alpha$ terminating at $\Lambda(x)$. The labeling map $\Lambda_\alpha : X \to f^{-1}(z_1)$ is defined by taking $\Lambda_\alpha(x)$ to be the beginning point of $\tilde{\alpha}_x$. Having specified the standard action at the basepoint $z_1$, we see that elements identified by the isomorphism $\alpha_*$ have equal actions on $X^*$.

Suppose that $g(x)=x$ for $g \in \pi_1(\rs\setminus P_f,z_0)$. Then the lift  of $g^\alpha$ based at $x$, denoted $\tilde{g^\alpha}$, satisfies 
\[
\tilde{g^\alpha}=\tilde{\alpha_x}^{-1}\tilde{g}_x \tilde{\alpha_x},
\] 
where $\tilde{g}_x$ is the unique lift of $g$ based at $\Lambda(x)$. A consequence of this definition is that if $g|_x=g$ and $g(x)=x$, then from Proposition \ref{imgaction} we have
\begin{align*}
g^\alpha|_x &=
(\tilde{\alpha}_x^{-1}\ell_x\alpha)^{-1}\tilde{g^\alpha}(\tilde{\alpha}_x^{-1}\ell_x\alpha)\\
&\simeq \alpha^{-1}\ell_x^{-1}\tilde{g}_x\ell_x\alpha\\
&=\alpha^{-1} g|_x\alpha\\
&\simeq
\alpha^{-1} g\alpha\\
&=g^\alpha.
\end{align*}
Extending to words of higher length using Equations \eqref{nested restriction}, we see that membership in $\mathcal{N}_1^\pi$ is unaffected by a change of basepoint.

Due to subhyperbolicity, the elements of $\mathcal{N}_1^{\pi}$ are very special. Recall the discussion of peripheral loops in Section \ref{subs:periphLoops}.

\begin{proposition}
\label{peripheral}
Each nontrivial element of $\mathcal{N}_1^{\pi}$ is  peripheral about a repelling periodic post-critical point.
\end{proposition}

\begin{proof}
Suppose that $g\in\mathcal{N}_1^\pi$ is nontrivial. By the remarks immediately preceding this proposition, we may assume that the basepoint of the fundamental group is in the compact subset $K'$ where the expansion of Equation \eqref{derivBound} holds. Choose a representative $\gamma$ of  $g$ so that $\gamma$ lies in $K'$. By hypothesis there exists a non-empty $w \in X^*$ where $g(w) = w$ and $g|_w =g$.  For $i \geq 1$, let $\gamma_i$ be the lift of $\gamma$ based at $\Lambda^*(w^i)$ where $w^i$ is the concatenation of $i$ copies of $w$. Since $g(w^i)=w^i$, Lemma \ref{liftable} implies that each $\gamma_i$ is a loop. Equation \eqref{restriction} implies that  $g|_{w^i}=[\ell_{w^i}^{-1}\gamma_i\ell_{w^i}]$, where there is an evident free homotopy $\ell_{w^i}^{-1}\gamma_i\ell_{w^i}\simeq \gamma_i$ in $\rs\setminus P_f$. Since $g|_{w^i}=g$ by hypothesis, it follows that there is a free homotopy $\gamma_i\simeq\gamma$ in $\rs\setminus P_f$. Each $\gamma_i$ is in the compact set $K'$ since $f^{-1}(K')\subset K'$, so Equation \eqref{derivBound} implies that the length of $\gamma_i$ converges to 0 as $i\to\infty$, and hence the curves $\gamma_i$ converge to a point $p \in \rs$. Because $g$ is nontrivial, $g$ has non-trivial restrictions at arbitrarily long words, and hence $p$ must be a periodic post-critical point. Each post-critical cycle of a PCF rational function either contains a critical point or is repelling. The compact set $K'$ was produced by deleting neighborhoods of the periodic critical cycles, and therefore $p$ is repelling. For large enough $i$, $\gamma_i$ is peripheral about $p$, and
because each $\gamma_i$ is freely homotopic to $\gamma\in g$, we conclude that $g$ is peripheral about the same point. 
\end{proof}

An immediate application of Proposition \ref{peripheral} is that   $\mathcal{N}_1^\pi$ is closed under passing to primitives:

\begin{corollary}\label{cor:PassToPrim}
If $g^m\in\mathcal{N}_1^\pi$ for some $g\in \pi_1(\rs\setminus P_f)$, then $g\in\mathcal{N}_1^\pi$.  
\end{corollary}
\begin{proof}
Let $w \in X^*$ be such that $g^m(w) = w$ and $g^m|_w = g^m$. Denote the length of $w$ by $|w|$. If $g^m$ is trivial in $\pi_1(\rs\setminus P_f)$, then so is $g$, whence $g\in\mathcal{N}_1^\pi$. Otherwise, by Proposition \ref{peripheral}, $g^m$ is peripheral about a repelling periodic point $p$. In the nontrivial case of $|P_f|>2$ the universal cover of $\rs\setminus P_f$ is the hyperbolic disk. The deck transformation corresponding to each peripheral loop is a parabolic element (a M\"obius transformation with exactly one fixed point), and the deck transformation corresponding to each nonperipheral loop is hyperbolic (a M\"obius transformation with exactly two fixed points). The power of a hyperbolic element is hyperbolic, so if $g^m$ is peripheral $g$ is also peripheral. Moreover the fixed set of the deck transformation corresponding to $g$ coincides with that of $g^m$, so $g$ must also be peripheral about $p$.  Since a repelling periodic point contains no critical point in its forward orbit, each iterate of $f$ is univalent on some neighborhood of $p$.  
Thus the lift of $g$ based at $\Lambda^*(w)$ is a loop so by Lemma \ref{liftable}, $g(w)=w$. 
Thus $g|_w=g^k$ for some $k\in\mathbb{Z}\setminus\{0\}$. 
The fact that $f^{|w|}$ is univalent and orientation preserving near $p$ implies that $k=1$. 

\end{proof}

\begin{remark} Each end of $X^*$ that is fixed by $g$ is also fixed by $g^m$. Thus if $g^m$ fixes only finitely many ends, so must $g$.
\end{remark}

Recall that a complex rational map is {\it dynamically exceptional} if there exists a finite, nonempty set $\Sigma$ with
$$f^{-1}(\Sigma)\setminus C_f = \Sigma,$$
where $C_f \subset \rs$ is the set of critical points of $f$. 
Let $p\in \Sigma$ and observe that every choice of a backward orbit of $p$ must intersect the critical set with only one possible exception: $p$ is contained in a periodic cycle (which necessarily contains no critical points, so will be a repelling cycle under forward iteration).

\begin{proposition}
\label{condition}
Suppose $f$ is a PCF rational map with an element $g\in\mathcal{N}_1^{\pi}$ that fixes only finitely-many ends of $X^*$.
Then $f$ is dynamically exceptional.
\end{proposition}


\begin{proof}
Since $g$ is clearly not trivial, Proposition \ref{peripheral} implies $g$ is peripheral about some post-critical point $p$ that is contained in a non-critical cycle. We may assume that $g$ is primitive and fixes only finitely many ends of $X^*$ by Corollary \ref{cor:PassToPrim} and the ensuing remark. 
Let $\Sigma\subset\rs$ be the set of points whose forward orbit contains $p$ but does not intersect $C_f$. Since $p$ lies in a non-critical cycle, $p\in\Sigma$ and so $\Sigma\neq\emptyset$. Because $g$ is primitive and peripheral, we may invoke the second part of Lemma \ref{periend} to conclude that every backward orbit of $p$ either intersects $C_f$ or is a subset of $P_f$. Thus $\Sigma\subset P_f$ and is hence finite. 

We now argue that $\Sigma= f^{-1}(\Sigma)\setminus C_f$.  Because $p$ is periodic, it follows that $f(\Sigma)\subset\Sigma$. Thus $\Sigma\subset f^{-1}(\Sigma)$ and since $\Sigma\cap C_f=\emptyset$, it follows that $\Sigma\subset f^{-1}(\Sigma)\setminus C_f$.
To see that $f^{-1}(\Sigma)\setminus C_f\subset\Sigma$, observe that if $x\in f^{-1}(\Sigma)\setminus C_f$ then $f(x)\in\Sigma$, and hence the forward orbit of $f(x)$ contains $p$. Thus the forward orbit of $x$ contains $p$, and since $x$ is not critical, $x\in\Sigma$. This proves that $f$ is dynamically exceptional.
\end{proof}

\subsection{End behavior of non-exceptional maps: IMG}


A sequence of elements $(g_n)_{n=0}^{\infty}$ in a group is said to be \emph{eventually periodic} (resp. \emph{eventually peripheral}) 
if there is some integer $N$ so that $(g_n)_{n=N}^{\infty}$ is periodic. (resp. peripheral) 
Note that periodic sequences are eventually periodic under this definition.

For any string $w\in X^n$ and a positive integer $m$, recall that we denote by $w^m$ the string in $X^{mn}$ formed by concatenating $m$ copies of $w$.

\begin{lemma}\label{lem:sequencePeriodic}
Suppose that there is $g\in\pi_1(\rs\setminus P_f)$ and a nonempty word $w$ so that $g(w^m)=w^m$ for all $m>0$. Then the sequence of restrictions $g_m:=g|_{w^m}$ is
eventually periodic.
\end{lemma}


\begin{remark}
For a PCF rational map $f$ it is known that $IMG(f)$ is contracting (Theorem \ref{cor:IMGnucleus}). Since the finite set $\mathcal{N} \subset IMG(f)$ of Definition \ref{contractingdef} is closed under restriction, the lemma clearly holds if ``$\pi_1(\rs\setminus P_f)$" is replaced with ``$IMG(f)$". However, the same argument cannot be used to prove Lemma \ref{lem:sequencePeriodic} because there is in general no finite set $\mathcal{N}$ as in Definition \ref{contractingdef} for $G = \pi_1(\rs\setminus P_f)$.  Consider for example the Chebyshev map $f(z) = z^2-2$, which has a repelling fixed point at $2$. Let $\alpha$ be a loop that is peripheral about 2. Observe that $f^{-1}(2) = \{\pm 2\}$, and so there is $x \in X$ such that $\alpha(x) = x$ and $\alpha|_x = \alpha$. Concatenating $x$ with itself $n$ times gives a word $w \in X^n$ with $\alpha(w) = w$ and $\alpha|_w = \alpha$. These same statements hold with $\alpha$ replaced by $\alpha^m$, and because the $\alpha^m$ are pairwise non-homotopic this gives rise to an infinite subset of $\pi_1(\rs\setminus P_f)$ that can occur as restrictions of arbitrarily long words. In conclusion, Lemma \ref{lem:sequencePeriodic} is not an immediate consequence of the existing theory.

\end{remark}


\begin{remark}
The following proof in fact shows that the sequence $g_m$ is eventually constant, rather than merely eventually periodic. However, eventual periodicity is sufficient for our purposes.
\end{remark}

\begin{proof}


Recall the construction of the backward-invariant compact set $K'$ where expansion holds. As with the proof of Proposition \ref{peripheral}, we may assume the basepoint $y_0$ for the fundamental group $\pi_1(\rs\setminus P_f)$ is in $K'$. Let $F=f^{|w|}$, and fix a representative $\gamma\subset K'$ of the class $g$. Let $\gamma_m:=\gamma|_{w^m}$. If there exists $m_0$ so that the homotopy class $[\gamma_{m_0}]$ is trivial, then $[\gamma_m]$ is trivial for all $m>m_0$, and hence $[\gamma_m]$ is eventually periodic (indeed, eventually constant). For the rest of the proof we assume that $[\gamma_m]$ is non-trivial for all $m$.

We recall the explicit construction of $\gamma_m$ via a standard action, as described in Definition \ref{standardactiondef}. We assume the paths $\{\ell_i\}$ in Definition \ref{standardactiondef} are selected to lie in $K'$. Because $g(w) = w$ it follows from Proposition \ref{imgaction} that $\gamma|_w = l_1^{-1} \tilde{\gamma}_w l_1$, where $\tilde{\gamma}_w$ is the lift of $\gamma$ starting at $\Lambda^*(w)$ and $l_1$ is a concatenation of lifts of the paths $\ell_x$, corresponding to the letters in the word $w$. Because $K'$ is backward invariant and each $\ell_x \subset K'$, we have $l_1 \subset K'$. Denote by $y_1$ the endpoint of $l_1$, which by Equation \eqref{Lstar} is the same as $\Lambda^*(w)$.



Now define the sequence $y_m:=\Lambda^*(w^m)$ $\in\rs$.  Let $l_i$ be the unique lift of $l_1$ under $F^{i-1}$ based at $y_{i-1}$, and observe that $l_i$ connects $y_{i-1}$ to $y_i$ and is contained in $K'$. Finally, let $\lambda_m$ be the concatenation of the paths $l_1,\dots, l_m$, where evidently $\lambda_m$ connects $y_0$ to $y_m$. Due to the geometric expansion of $F$ on $K'$ in the orbifold metric from equation \eqref{derivBound}, the lengths of the paths $l_m$ decrease geometrically. Hence the sequence $(y_i)$ is Cauchy and converges to a point $p\in\rs$. Moreover, the length of $\lambda_m$ is uniformly bounded and so $\lambda_m$ converges to a path $\lambda_\infty$ of finite length that connects $y_0$ to $p$. 

The continuity of $F$ and the equation $F(y_i)=y_{i-1}$ imply that $F(p)=p$. Let $\alpha_m$ be the unique lift of $\gamma$ under $F^m$ based at $y_m$. The hypothesis that $g(w^m)=w^m$ together with Lemma \ref{liftable} imply that $\alpha_m$ is a loop and so $\gamma_m=\lambda^{-1}_m\alpha_m\lambda_m$ for each $m$. By  \eqref{derivBound}, the length of $\alpha_m$ converges to 0, so $\alpha_m$ is arbitrarily small for large $m$. We have already dispensed with the case that $\alpha_m$ is homotopically trivial, thus it follows that $\alpha_m$ is eventually peripheral about $p\in P_f$.  

Since both $\alpha_m$ and $\lambda_\infty\setminus\lambda_m$ have length converging to zero, 
for each disk of radius $\epsilon$ about $p$ (denoted $D_\epsilon(p)$) there exists an integer $N$ so that for $m>N$, the paths $\gamma_m$ and $\gamma_{m+1}$ coincide on the complement of $D_\epsilon(p)$ up to reparametrization. Fix $\epsilon$ so that $D_\epsilon(p)\cap P_f\setminus\{p\}=\emptyset$ and  $\alpha_m\subset D_\epsilon(p)$ for all $m>N$. 
Since $F$ maps $\alpha_{m+1}$ to $\alpha_m$ with degree 1, we have that the loops $\alpha_m$ and $\alpha_{m+1}$ are freely homotopic in $D_\epsilon(p)\setminus\{p\}$. We thus have two peripheral loops $\gamma_m$ and $\gamma_{m+1}$ that agree outside of $D_\epsilon(p)$ and are both freely homotopic to the same curve in $D_\epsilon(p)$. Therefore there is a based homotopy between $\gamma_m$ and $\gamma_{m+1}$, showing that $g_m = g_{m+1}$. 
\end{proof}




\begin{proposition}
\label{modok}
Let $f$ be a PCF rational function.
Then some element of $\mathcal{N}_1^{\pi}$ fixes only finitely many ends of $X^*$ if and only if some element of $\mathcal{N}_1$ fixes only finitely many ends of $X^*$.
\end{proposition}

\begin{proof}
Recall from Definition \ref{imgdef} that $\IMG(f)$ is the quotient of $\pi_1(\rs \setminus P_f)$ by the faithful kernel $K$ of the monodromy action on $X^*$. So if $g\in\mathcal{N}_1^{\pi}$ fixes only finitely-many ends of $X^*$, then its image under the quotient is an element of $\mathcal{N}_1$ that fixes only finitely-many ends of $X^*$.

Now assume there is an element $\bar{g}\in\mathcal{N}_1$ that fixes only finitely-many ends of $X^*$.
It follows from the definition of $\mathcal{N}_1$ that there is a finite string $w\in X^n$ for some $n\geq 1$ so that $\bar{g}(w)=w$ and $\bar{g}|_w=\bar{g}$.
Let $g\in\pi_1(\rs\setminus P_f)$ be in the coset of $K$ 
represented by $\bar{g}$.
 Then for each $m \geq 1$ we have $g(w^m)=w^m$, but we only know that $g|_{w^m}$ and $g$ lie in the same coset of $K$.

Define the sequence $g_m := g|_{w^m}$, observing that each $g_m$ fixes only finitely many ends of $X^*$. 
It follows from Lemma \ref{lem:sequencePeriodic} that $g_{m_1} = g_{m_2}$ for some $m_1\neq m_2$.  Then the restriction of $g_{m_1}$ to $w^{|m_2-m_1|}$ is $g_{m_2}$ (indeed, by the second remark before the proof of Lemma \ref{lem:sequencePeriodic}, we may take $m_2 - m_1 = 1$). This proves that $g_{m_1}\in\mathcal{N}_1^{\pi}$.
\end{proof}

\begin{proof}[Proof of Theorem \ref{repeat_maincx2} (a.k.a. Theorem \ref{maincx2})]
Let $f$ be a PCF rational map that is not exceptional. The contrapositive of Proposition  \ref{condition} guarantees that each element of $\mathcal{N}_1^\pi$  fixes infinitely many ends. Then Proposition \ref{modok} implies that each element of $\mathcal{N}_1$ fixes infinitely many ends.
\end{proof}

\subsection{Characterization of exceptional maps}

A characterization of dynamically exception maps is given in Theorem \ref{thm:ExceptionalChar}, though this result is not used elsewhere in this paper. The result is easily proved if the set $\Sigma$ contains a fixed point, but the presence of higher period cycles requires some minor technicality about passing to iterates.

Recall the construction of the standard tree $X^*$ from Section \ref{subs:IMGs} in terms of the labeling map 
\begin{equation}\label{eqn:labelingMap}
    \Lambda:X=\{0,\dots d-1\}\to f^{-1}(z_0).
\end{equation} In principle, one could use the construction of that section to associate a standard action to $f^n$ using a labeling map $\{0,\dots,d^n-1\}\to f^{-n}(z_0)$. However, we choose to use a labeling that is compatible with the standard action induced by $\Lambda$ in  \eqref{eqn:labelingMap}. Specifically, our new labeling map 
\[\Lambda_n:X^n\to f^{-n}(z_0)\] is defined for a given point $w\in X^n$ by $\Lambda_n(w)=\Lambda^*(w)$, where $\Lambda^*$ is the extension of $\Lambda$ to elements of $X^*$ described in \eqref{Lstar}. In this way, the point $\Lambda_n(w)\in f^{-n}(z_0)$ is labeled by a string of $n$ characters in the alphabet $X$, even though it is a ``first-level" preimage of $z_0$ under $f^n$. Define the connecting path for $z\in f^{-n}(z_0)$ to be $\ell_{\Lambda_n(z)}$. This data defines a tree isomorphism from the preimage tree $T_{f^n,z_0}$ to a standard $d^n$-ary tree which we denote $(X,f^{n})^*$, as well as a standard action by $\pi_1(\rs\setminus P_{f^{n}})$. Since $P_{f^{n}} = P_f$, we have that \[\pi_1(\rs\setminus P_{f^{n}}) = \pi_1(\rs\setminus P_f).\]

Using this newly defined standard action, we may now define the iterated analogue of Equation \ref{eqn:piN1}:
\[
\mathcal{N}_1^{\pi}(f^{m}) := \{g\in \pi_1(\rs\setminus P_f)\ :\ \exists \text{ nontrivial } w\in (X,f^{m})^* \text{ so that } g(w) = w \text{ and }g|_w = g \}
\]

\begin{lemma}
\label{iterateok}
Let $f$ be a PCF rational map, and let $m\geq 1$.
Then $\mathcal{N}_1^{\pi}(f^{m})\subset \mathcal{N}_1^{\pi}$.
\end{lemma}

\begin{proof}

Let $g\in \mathcal{N}_1^{\pi}(f^{m})$.
Then there exists nontrivial $w\in (X,f^{m})^*$ so that $g(w)=w$ and $g|_w = g$. By construction $\Lambda^*(w)=\Lambda_m(w)$. Since $f^{m|w|}=(f^m)^{|w|}$, Proposition \ref{imgaction} implies that the action of $g$ on $w$ is independent of whether $w$ is a vertex in $X^*$ or $(X,f^m)^*$. Likewise, Equation \ref{restriction} of Proposition \ref{imgaction} implies that $g|_w$ is independent of whether $w$ is a vertex in $X^*$ or $(X,f^m)^*$. Thus considering $w$ now as an element of $X^*$, we have that  the standard action on $X^*$ satisfies $g(w)=w$ and $g|_w = g$.

\end{proof}

\begin{proposition}
\label{exceptional}
Let $f$ be a dynamically exceptional map that is PCF. Then for some $n$, there exists $g\in\mathcal{N}_1^{\pi}(f^{\circ n})$ that fixes only finitely-many ends of $(X,f^{\circ n})^*$.
\end{proposition}
\begin{proof}
Recall that for a dynamically exceptional map, the set $\Sigma$ satisfies $f(\Sigma)\subset\Sigma$, so there must be some point $p\in\Sigma$ that is periodic. By the defining property of $\Sigma$, the point $p$ cannot lie in a critical cycle. Since $f$ is PCF, $p$ must then be repelling. Passing to an iterate, we assume that $p$ is fixed. Let $\lambda:=f'(p)$. 

Recall that fixed repelling periodic points are linearizable \cite[Thm 8.2]{MilnorBook}, namely there is a univalent holomorphic change of coordinates $\phi(z)=w$ on some neighborhood $U$ of $p$ so that $\phi(p)=0$ and $\phi\circ f\circ \phi^{-1}=\lambda w$. Choose $U$ so that $f(U)$ intersects the post-critical set only at $p$ (this is possible since $f$ is PCF). Let $A$ be the preimage under $\phi$ of a fundamental annulus in coordinates. Then $\partial A$ consists of two topological circles $C$ and $C'$ with $f(C')=C$.

Fix a basepoint $z\in C$ and an orientation on $C$. Let $g$ be a loop based at $z$ that winds once around $p$ (i.e. is primitive) and respects the orientation. Let $g'$ be the unique lift of $g$ contained in $C'$, where evidently the map $g'\to g$ is univalent. Let $z'\in C'$ be the unique preimage of $z$ under this map. Let $\ell_{z'}$ be some choice of connecting path in $A$ that joins $z$ to $z'$. The path $\ell_{z'}^{-1} g'\ell_{z'}$ is a loop in $A$ based at $z$. Using the annular coordinates defined by $A\subset \rs\setminus P_f$, it can be shown that $\ell_{z'}^{-1} g' \ell_{z'}$ is homotopic to $g$ relative to the basepoint. Since $f$ is orientation preserving, $g$ and $g'$ have the same orientation. Let $w$ be the label of the point $z'$, i.e. $\Lambda(w)=z'$. Then from what was just argued, $g|_w=g$. By the univalence of $g'\to g$, it follows that $g(w)=w$.

Since $f$ is dynamically exceptional and $p$ is fixed, any backward orbit other than the constant one at the fixed point $p$ will meet a critical point.
By Lemma \ref{perilift}, the only end of $X^*$ that the action of $g$ will fix is $w^\infty$.
\end{proof}

\begin{theorem}\label{thm:ExceptionalChar}
A PCF complex rational map $f$ is dynamically exceptional if and only if there is an element $g\in\mathcal{N}_1$ that fixes only finitely many ends of $X^*$.\end{theorem}

\begin{proof}
Suppose that $f$ is dynamically exceptional and PCF.
Then by Proposition \ref{exceptional}, there is an element of $\mathcal{N}_1^{\pi}(f^{m})$ that fixes only finitely-many elements of $(X,f^{m})^*$.
By Lemma \ref{iterateok}, this element is also an element of $\mathcal{N}_1^{\pi}$. Proposition \ref{modok} guarantees existence of an element in $\mathcal{N}_1$ that fixes only finitely-many ends of $(X,f^{m})^*$, and by the identification of the ends of $(X,f^{m})^*$ with the ends of $X^*$, it only fixes finitely-many ends of $X^*$ as well.

Suppose now instead that $f$ is a PCF rational map such that there is an element of $\mathcal{N}_1$ that fixes only finitely-many ends.
Then by Proposition \ref{modok} there is an element of $\mathcal{N}_1^\pi$ that fixes only finitely-many ends.
By Proposition \ref{condition}, the map is dynamically exceptional.
\end{proof}

\bibliographystyle{plain}

\end{document}